\documentclass[11pt]{amsart}

\usepackage{amssymb}
\usepackage{color}
\def\m{{\mathfrak m}} 

\def\SZ{\mathbb Z} 

\def\ff{\mathfrak}
\def\spec{\mbox{\rm spec}}
\def\Spec{\mbox{\rm Spec}}

\def\xpd{\mbox{\rm xpd}}
\def\Ass{\mbox{\rm Ass}}
\def\cal{\mathcal}

\newtheorem{thm}{Theorem}[section]
\newtheorem{cor}[thm]{Corollary}
\newtheorem{prop}[thm]{Proposition}
\newtheorem{lem}[thm]{Lemma}

\newtheorem{rem}[thm]{Remark}
\newtheorem{definition}[thm]{Definition}

\begin{document}

\author{Bruce Olberding}


\address{Department of Mathematical Sciences, New Mexico State University,
Las Cruces, NM 88003-8001 \textup{: \texttt{olberdin@nmsu.edu}}}
\thanks{The first author thanks Universit\'a degli studi ``Roma Tre'' for its hospitality and support for  a visit in which this work was begun.}

\author{Francesca Tartarone}

\address{Dipartimento di Matematica,  Universit\'a degli Studi ``Roma Tre','   Largo San Leonardo Murialdo 1, Roma 00146 - Italy, \textup{: \texttt{tfrance@mat.uniroma3.it}}}

\title[Integrally closed rings in birational extensions]
  {Integrally closed rings in birational extensions of two-dimensional regular local rings}



\maketitle

\begin{abstract}
Let $D$ be an integrally closed local Noetherian domain of Krull dimension $2$, and let $f$ be a nonzero element of $D$ such that $fD$ has prime radical.  We consider when an integrally closed ring $H$ between $D$ and $D_f$ is determined locally by finitely many valuation overrings of $D$.  We show such a local determination  is equivalent to a  statement about the exceptional prime divisors of normalized blow-ups of $D$, and, when $D$ is analytically normal,  this property holds for $D$ if and only if it holds for the completion of $D$.
This latter fact, along with MacLane's notion of key polynomials, allows us to prove
that in some central cases where $D$ is a regular local ring and $f$ is a regular parameter of $D$, then $H$ is determined locally by a single valuation.
 As a consequence, we show that if $H$ is also the integral closure of a finitely generated $D$-algebra, then the exceptional prime ideals of the extension $H/D$ are comaximal. Geometrically, this translates into a statement about intersections of irreducible components in the closed fiber of the normalization of a proper birational morphism.
\end{abstract}

\section{Introduction}

In this article we consider integrally closed birational extensions
of two-dimen\-si\-on\-al local Noetherian domains from two perspectives.
In  the first perspective, which is a ``top-down'' approach, we use
the fact that every integrally closed domain is the intersection of
its valuation overrings to examine how such rings can be represented
as intersections of valuation rings.  In the second perspective, we
employ a ``bottom-up'' approach and view an integrally closed birational extension as a
direct limit of normalized blow-up algebras.  The specific
phenomenon we are interested in is when finitely many valuations,
along with  a flat overring, serve to determine the
birational extension locally, and we show that this is equivalent, in the approach to the
ring from underneath via blow-up algebras, to the existence of a bound on how many irreducible components of an affine piece of the closed fiber of a normalized blow-up meet in an arbitrary point.

One of our main motivations is the open problem of classifying the not-necessarily-Noetherian integrally closed rings    between a two-dimensional Noetherian domain $D$ and its quotient field; see \cite{OlbS} for a survey of this topic and \cite{lt} for a classification of the integrally closed rings between  ${\mathbb{Z}}[X]$ and ${\mathbb{Q}}[X]$.    A more modest version of this problem is to describe the integrally closed rings between $D$ and a given finitely generated $D$-subalgebra of the quotient field of $D$; in other words, we wish to describe the integrally closed rings between $D$ and $D[\frac{g_1}{f},\ldots,\frac{g_n}{f}]$, when $f,g_1,\ldots,g_n \in D$. This problem is encompassed by the more general one of describing the integrally closed rings $H$ between $D$ and $D_f:=D[1/f]$, where $0\ne f \in {\ff m}$.
  A first step in this direction is to determine the local structure of such rings $H$.
Specifically, in the present article we are interested in when for each maximal ideal $M$ of $H$, there exist {\it finitely many} valuation overrings $V_1,\ldots,V_n$ of $H$ such that $H_M = V_1 \cap \cdots \cap V_n \cap (D_f)_M$ (here, as throughout the article,  the localization of $D_f$ is with respect to the set $H\setminus M$).

Although our methods are purely algebraic, we also pose a geometric version of this phenomenon  to give further context. 
To frame the geometrical version, 
recall that a morphism $\pi:X \rightarrow S$ of schemes is a {\it modification} if $\pi$ is proper and birational.
 For example, when  $S = \Spec(D)$, with $D$ a domain, $u$ an indeterminate for $D$ and $I$  an ideal of $D$, then the blow-up $\pi: $ Proj$(D[Iu]) \rightarrow \Spec(D)$, as a projective birational morphism, is a modification.  When $D$ is a quasilocal domain with maximal ideal ${\ff m}$, and $\pi:X \rightarrow \Spec(D)$ is a morphism of schemes, then $\pi^{-1}({\ff m})$ is the {\it closed  fiber} of $\pi$. (More precisely, the closed fiber is $X \times_{{\rm{Spec}}(D)} \Spec({\Bbbk})$, where ${\Bbbk}$ is the residue field of $D$,  but $\pi^{-1}({\ff m})$ can be viewed as the 
  underlying topological space of this scheme.) For $0 \ne f \in D$, we denote by $X_f$ the open subscheme of $X$ consisting of all the points $x \in X$ such that $f$ is not in the maximal ideal of the local ring ${\cal O}_{X,x}$ of the point $x$. 
 For a morphism $\pi:X \rightarrow S$, we denote by  $\overline{\pi}:\overline{X} \rightarrow S$ the normalization of $\pi$.  
In our setting, we are interested in the intersection of irreducible components of the closed fiber of the normalization of a modification of $\Spec(D)$, where $D$ is a two-dimensional integrally closed local Noetherian domain. 
 The following theorem, which will be proved  at the end of  Section 4, summarizes the connections between the above ideas.

\begin{thm}  \label{first theorem}
Let $(D,{\ff m})$ be a two-dimensional integrally closed local Noetherian domain, let $0 \ne f \in {\ff m}$ and let $n$ be a positive integer.  Then the following are equivalent. 
\begin{itemize}

\item[(1)] For each 
modification $\pi:X \rightarrow \Spec(D)$,  at most $n$
irreducible components  of the closed  fiber of the normalization of $\pi$  meet  in any  affine open subscheme of $\overline{X}$ containing $\overline{X}_f$. 

\item[(2)] For every finitely generated $D$-subalgebra $H$ of $D_f$, at most $n$  of the  height $1$ prime ideals of the integral closure $\overline{H}$ of $D$ lying over ${\ff m}$  are contained in any single maximal ideal of $H$.

\end{itemize}
If also $\sqrt{fD}$ is a prime ideal of $D$, then (1) and (2) are equivalent to  
\begin{itemize}
\item[(3)]  For every integrally closed ring $H$ between $D$ and $D_f$ and maximal ideal $M$ of $H$, there is a representation $H_M = V_1 \cap \cdots \cap V_n \cap (D_f)_M,$ for some not necessarily distinct valuation overrings $V_1,\ldots,V_n$ of $H$.

\end{itemize}
\end{thm}

When also $D$ is analytically normal and $f$ has prime radical in the completion of $D$, Theorem~\ref{pass to completion} shows that whether $D$ satisfies (1)--(3) is equivalent to whether the ${\ff m}$-adic completion of $D$  does.  This observation is one of our key  tools in exhibiting a choice of $D$, $n$ and $f$ that satisfies (1)--(3).  Under the assumption that $D$ is a regular local ring and $f$ is a regular parameter, we use this reduction  in a strong way in Section 6
to pass to the case where $D = V[X]$, with $V$ a DVR.  This case then hinges on technical arguments in the longest section of the paper, Section 7, involving MacLane's extension of valuations via key polynomials.  With this case settled,    we prove in
Section 6 our main result:

\begin{thm} \label{second theorem}  {Let $D$ be a regular local ring of Krull dimension $2$,
and let $f$ be a regular parameter of $D$.
Suppose that either (a) $D$ is equicharacteristic, or (b) $D$ has mixed characteristic and $f$ is a prime integer in $D$.
 Then $D$, $f$ and $n=1$ satisfy the equivalent conditions (1)--(3) of Theorem~\ref{first theorem}.
}
\end{thm}

In particular, as an application to non-Noetherian commutative ring theory,  this gives a local classification of   integrally closed rings between $D$ and $D_f$:   

\begin{cor}
 With the same assumptions as Theorem~\ref{second theorem},
for any integrally closed ring $H$ between $D$ and $D_f$ and maximal ideal $M$ of $H$, there exists a valuation overring $V$ of $H$ such that $H_M = V \cap (D_f)_M$.
\end{cor}

  Thus $H$ is locally as simple as one could hope for: It is locally the intersection of a valuation overring and the flat PID overring $(D_f)_M$. In a future article, we work out the consequences for the structure of the ring $H$ based on this representation. 
    Such behavior cannot be expected without some strong restrictions on $f$: A two-dimensional integrally closed local Noetherian domain $D$  is expressible with such a (necessarily local, since $D$ is local) representation as in the theorem above if and only if $f$ is contained in a unique height $1$ prime ideal of $D$. Thus when the radical of $fD$ is not a prime ideal, 
 statement (3) of the theorem cannot be satisfied  with the choice of $n=1$.   


In any case, in the special setting of Theorem~\ref{second theorem}, we have the immediate consequence for affine $D$-algebras:

\begin{cor}  With the same assumptions as Theorem~\ref{second theorem}, if  $H$ is the integral closure of a finitely generated $D$-subalgebra of $D_f$, then distinct height one prime ideals of $H$ lying over the maximal ideal of $D$ are comaximal.

\end{cor}

The prime ideals in the corollary correspond to Rees valuation rings for an ideal of $D$, as we recall in Section 2, and we mention an application of this for one-fibered ideals in Corollary~\ref{one-fibered corollary}.

 Finally, returning to the geometric interpretation, 
we conclude:

\begin{cor} \label{connectedness}  With the same assumptions as Theorem~\ref{second theorem}, let 
 $\pi:X \rightarrow \Spec(D)$ be a modification, 
with $X$ normal.  
If the closed fiber of ${\pi}$ has more than one irreducible component, then each irreducible component must meet another irreducible component of the closed fiber, but not in any affine open subset of ${X}$ containing ${X}_f$.  
 \end{cor}

The corollary follows from Theorem~\ref{second theorem} and Grothendieck's version of Zariski's Connectedness Theorem, which implies that the fibers of a modification $X \rightarrow S$ of Noetherian integral schemes, with $S$ normal, are all connected \cite[(4.3.1), p.~130]{EGAIII}. 

If also $D$ is excellent, then the normalization of a scheme of finite type over $\Spec(D)$ is finite, so that the normalization of a modification is again a modification, since the composition of a finite morphism with a proper morphism is again proper \cite[Exercise 4.1, p.~105]{Hart}. Thus when $D$ is excellent the corollary may be restated for normalizations of modifications.

  We do not know  whether  for a two-dimensional integrally closed domain $D$ and nonzero element $f$ in $D$, there exists a universal bound $n>0$ depending only on $D$ and $f$ such that (1), (2) or (3) of Theorem~\ref{first theorem} hold. In fact, the only situation  of which we are aware in which  (1), (2) or (3) hold for a universal bound $n$ depending only on $D$ and $f$ is in the special  setting of Theorem~\ref{second theorem} with $n=1$. 

\section{Preliminaries}

In this section we discuss some  basic terminology  and
properties involving valuation domains  and extensions of
two-dimensional Noetherian domains.


An {\it overring} of a domain $D$ is a ring between $D$ and its quotient field.  A quasilocal domain $R$ {\it dominates} a quasilocal domain $D$ if the maximal ideal of $D$ is a subset of the maximal ideal of $R$.  Throughout the paper, when $D$ is a quasilocal domain, we reserve the notation ${\ff m}$ for its maximal ideal.

 Given a field $K$ and a valuation $v$  on $K$  (with values in a
totally ordered group  $\Gamma$), the set of all the elements of $K$
assuming positive value through $v$ is a domain $V$ with quotient
field $K$ called the \textit{valuation
 domain}  associated to $v$.
Valuation domains are quasi-local, integrally closed and any integrally closed domain is the intersection of
its valuation overrings.

The set of   ideals of a valuation domain is
totally ordered under inclusion, that is if $I$ and $J$ are two
ideals of $V$, either $I \subseteq J$ or $J \subseteq I$.

A \textit{rank-one  discrete valuation domain (DVR)} is a
one-dimensional valuation domain with principal maximal ideal. A
valuation domain is Noetherian if and only if it is a DVR.

Any overring of a valuation domain $V$ is a valuation domain and it
is also a localization of $V$ at some prime ideal. Conversely,
each localization of $V$ at a prime ideal is a valuation overring of
$V$. Thus there is a bijecton between the prime spectrum of
$V$ and its overrings. In particular, if $V$ is one-dimensional, it has just one overring, which is $K$ (in fact, $K = V_{(0)}$).


\textit{Pr\"ufer domains} are the most strict generalization of valuation domains. In fact, one of their numerous characterizatons  states that $D$ is a Pr\"ufer domain if and only if the localization $D_P$ is a valuation domain, for each prime ideal  $P$ of $D$. Thus, valuation domains are exactly the quasi-local Pr\"ufer domains.

Each overring $T$ of a Pr\"ufer domain $D$ is an intersection of its localizations at some prime ideals, that is $T = \bigcap_{P \in \mathcal{P}}D_P$, where $\mathcal{P}$ is a subset of $\Spec(D)$. Thus the overrings of a Pr\"ufer domain $D$ are in one-to-one correspondence with the families of prime ideals of $D$, so again generalizing what happens for valuation domains.

Dedekind domains are exactly the Pr\"ufer, Noetherian domains.

In commutative algebra the Pr\"ufer property and some of its generalizatons (Pr\"ufer-like properties) have an important role in multiplicative ideal theory (for instance, B\'ezout domains are exactly the Pr\"ufer domains with trivial ideal class group).

For a more detailed reference about valuation theory and Pr\"ufer
domains we suggest to read, among others, \cite{gilmer}.

\medskip

{\bf (2.1)} {\it Hidden prime divisors.}
A valuation overring $V$ of a local Noetherian domain $D$ of Krull dimension $d$ is a {\it hidden prime divisor} of $D$ if $V$ is a DVR that dominates $D$ and the residue field of $V$ has transcendence degree $d-1$ over the residue field of $D$.  Since a valuation overring of $D$ has transcendence degree at most $d-1$
\cite[Theorem 1]{Abh}, it follows that
 when $D$ has Krull dimension $2$, a DVR overring that dominates $D$ is a hidden prime divisor if and only if its residue field is not algebraic over the residue field of $D$.
\medskip

{\bf (2.2)}  {\it Exceptional prime divisors.}  Let $H$ be an integrally closed Noetherian overring of the local Noetherian domain  $D$.
Following \cite{HL}, we say
the {\it exceptional prime divisors} of the extension $H/D$  are the hidden prime divisors of $D$ of the form $H_P$, where $P$ is a height $1$ prime ideal of $H$.   We denote the set of all exceptional prime divisors of $H/D$ by $\xpd(H/D)$. The exceptional prime divisors are localizations of $H$ at the ``exceptional prime ideals'' of $H/D$, which we define in (2.3).
The terminology here is motivated by a geometric interpretation:
 When $Y$ is the normalized blow-up of $\Spec(D)$ along a zero-dimensional closed subscheme and $U=\Spec(R)$ is an affine open subset of $Y$, then the exceptional prime divisors of $R/D$ are the local rings of the points in $U$ that are  generic points of  irreducible components  of the exceptional fiber of the blow-up.

\medskip

{\bf (2.3)}  {\it Exceptional prime ideals.}
The prime ideals $P$ arising as in (2.2) are the {\it exceptional prime ideals} of the extension $H/D$.  When $D$ is a local
 Noetherian domain of Krull dimension $2$, and  $H$ is a proper  overring of $D$ that is the integral closure of a finitely generated $D$-subalgebra of $H$, then for each  exceptional prime ideal $P$ of the extension $H/D$, the ring $H/P$ has Krull dimension $1$.
For
since $H$ has Krull dimension at most $2$, this amounts to the claim
that $P$ is a nonmaximal prime ideal of $H$, which in turn is a
consequence of Evans' version of Zariski's Main Theorem. Indeed, since $B \ne D$ and
 $B$ is a finitely generated algebra over the integrally closed local domain $D$, no
 prime ideal of the integral extension $H$ of $B$ containing ${\ff m}$ is both maximal and minimal among those that
 contain ${\ff m}$ \cite[Theorem, p.~45]{Evans}.
 So the height $1$ prime
 ideal $P$ of $H$ must be contained in another prime ideal of $H$, which proves that
 $P$ is nonmaximal.

\medskip

{\bf (2.4)} {\it Rees Valuations.}  The exceptional prime divisors can be used to define the Rees valuation rings of an ideal of $D$.  Let $I = (a_1,\ldots,a_n)$ be an ideal of $D$.   Then the set of {\it Rees valuation rings} of $I$ is $${\xpd}\left(\overline{D[Ia_1^{-1}]}\right) \cup \cdots \cup {\xpd}\left(\overline{D[Ia_n^{-1}]}\right).$$  Here, as throughout the paper, $\overline{R}$ denotes the integral closure of the domain $R$.
Thus the Rees valuation rings are the local rings of the generic points of the  irreducible components in the closed fiber of the  normalized blow-up of $\Spec(D)$ along $\Spec(D/I)$.
  In the case where $D$ has Krull dimension $2$, $I = (f,g)$ and $I$ is ${\ff m}$-primary, then the set of Rees valuation rings of $I$ is $\xpd(\overline{D[f/g]})$; i.e., you only need one affine piece of the normalized blow-up along  $\Spec(D/I)$ to capture all the Rees valuation rings \cite[p.~285]{HL}.  This is true in general for parameter ideals in arbitrary dimension; see the discussion on p.~438 of \cite{Sally}.
  \medskip

{\bf (2.5)} {\it One-fibered ideals.}
  An ideal $I$ of the local ring $D$ that has only one Rees valuation ring is  {\it one-fibered}.
For more on such ideals, see \cite{Sally} and \cite{Swanson} and their references.  The property of being one-fibered can be expressed also in terms of exceptional prime divisors of an extension.  Let $I = (f_1,\ldots,f_n)$ be an ideal of $D$, and for each $i$, let $H_i$ be
 the integral closure of $D[If_i^{-1}]$.  Then $I$ is one-fibered if and only if for each $i$, the extension $H_i/D$ has a unique exceptional prime divisor.  Moreover, if the dimension of $D$ is $n$ (the number of generators of $I$), so that $I$  is   primary for the maximal ideal of $D$, then $I$ is one-fibered if and only if there exists some $i$ such that $H_i/D$ has a unique exceptional prime divisor.  This is because, as discussed in (2.4), in such a case every Rees valuation ring of $I$ occurs as an exceptional prime divisor of the extension $H_i/D$, for any choice of $i$.
 Geometrically,
the ideal $I$ is one-fibered if and only if the closed fiber of the normalized blow-up of $\Spec(D)$ along $\Spec(D/I)$
is  irreducible.

\medskip

To phrase our main results in later sections, we introduce  the terminology of essentially $n$-fibered extensions and
essentially $n$-valuated subrings.  These are extensions which become $n$-fibered or $n$-valuated after passage to a localization:

\setcounter{thm}{5}

\begin{definition} \label{ess n-fibered}
If $D$ is a local Noetherian domain and $B$ is a finitely generated $D$-subalgebra of the quotient field of $D$, then
  the extension $B/D$ is {\it essentially $n$-fibered}   if for each prime ideal $P$ of $\overline{B}$, the extension $\overline{B}_P/D$ has at most $n$ exceptional prime ideals.
  \end{definition}

Thus $B/D$ is essentially $n$-fibered if and only if
at most $n$ irreducible components of the closed fiber of  $\Spec(\overline{B}) \rightarrow \Spec(D)$ intersect in any point in  $\Spec(\overline{B})$.
Whether $B$ is essentially $n$-fibered is determined entirely by its normalization $\overline{B}$.


\begin{definition}
     Let $D$ be a domain, and let $R$ be an  overring of $D$.   If $H$ is a ring with $D \subseteq H \subseteq R$, then we say that $H$ is an {\it essentially $n$-valuated subring of $R$}
     if for each prime ideal $P$ of $\overline{H}$, there exist (not necessarily distinct)  valuation overrings $V_1,V_2,\ldots,V_n$ of $H$ such that $\overline{H}_P = V_1 \cap V_2 \cap \cdots \cap V_n \cap R_P$.  (By $R_P$, we mean the localization of $R$ with respect to the multiplicatively closed subset $\overline{H}\setminus P$.)
     \end{definition}

      As in Definition~\ref{ess n-fibered} the property of being essentially $n$-valuated is determined by the normalization of $H$.
  Since the localization of a valuation ring is a valuation ring, it follows that whether $H$ is an essentially $n$-valuated subring of $R$ is determined by the maximal ideals of $\overline{H}$, in the sense that $H$ is an essentially $n$-valuated subring of $R$ if and only if for each maximal ideal $M$ of $\overline{H}$, there exist  valuation overrings $V_1,V_2,\ldots,V_n$ of $\overline{H}$ such that $\overline{H}_M = V_1 \cap V_2 \cap \cdots \cap V_n \cap R_M$.

\medskip

We mention next  in (2.8) and (2.9) two strong facts regarding
 integrally closed overrings of two-dimensional Noetherian domains

\medskip

{\bf (2.8)} {\it Integrally closed Noetherian overrings.}  In \cite{H}, Heinzer proves that an overring of a two-dimensional  Noetherian domain $D$ is Noetherian and integrally closed if and only if it is a Krull domain.  Thus the integrally closed Noetherian overings of $D$ are precisely those overrings that can be represented by a finite character intersection of DVRs.

\medskip

{\bf (2.9)} {\it Integrally closed local Noetherian overrings.}     If $D$ is a two-dimensional local Noetherian domain  that is analytically normal (i.e., its completion is a normal domain), then
every integrally closed local Noetherian overring of $D$ is the localization of a finitely generated $D$-algebra.  This can be deduced from a criterion for spots given by Heinzer, Huneke and Sally in \cite{HHS}.  For
since $D$ has Krull dimension $2$ and is analytically normal,  every integrally closed domain in the quotient field of $D$ which dominates $D$ and is the localization of a finitely generated $D$-algebra is analytically irreducible
 \cite[Proposition, p.~160]{Lipman}.  Then by \cite[Theorem 1]{HHS}, every normal local Noetherian domain $H$ in the quotient field of $D$ that dominates $D$ is the localization of a finitely generated $D$-algebra.   On the other hand, if $H$ is a normal Noetherian overring of $D$ that does not dominate $D$, then the maximal ideal of $H$ contracts to a height $1$ prime ideal ${\ff p}$ of $D$, and it follows since $D_{\ff p}$ is a DVR that $H = D_{\ff p}$.

\section{Essentially $n$-valuated subrings}

In this section we prove some general facts about essentially  $n$-valuated subrings which do not require restriction to
two-dimensional Noetherian domains.  These are technical results which will be useful later for
passing from representations of Noetherian rings between $D$ and
$D_f$ to representations of arbitrary integrally closed rings
between $D$ and $D_f$. The first two results use the ultraproduct
construction.  Let $\{B_\alpha:\alpha \in {\cal A}\}$ be a collection of rings, and let ${\cal U}$ be an ultrafilter on ${\cal A}$ (i.e., ${\cal U}$ is a filter on the power set of $\cal A$ that is maximal with respect to not containing the empty set).  The {\it ultraproduct} of the rings $B_\alpha$ with respect to ${\cal U}$ is the ring $\prod_{\cal U}B_\alpha := (\prod_{\alpha \in {\cal A}}B_\alpha)/I$, where $I = \{(b_\alpha):\{\alpha:b_\alpha = 0\} \in {\cal U}\}$. If all of the $B_\alpha$ are subrings of a ring $S$, and $R$ is also a subring of $S$, then by identifying $S$ with its image in $\prod_{\cal U} S$ under the diagonal mapping and $\prod_{\cal U}B_\alpha$ with its image in $\prod_{\cal U}S$, we may consider the subring $(\prod_{\cal U} B_\alpha) \cap R$ of $S$. Applying the relevant definitions, this subring consists of all $r \in R$ such that $\{\alpha:r \in B_\alpha\} \in {\cal U}$.

   \begin{lem} \label{ultraproduct}  Let $S$ be a ring, and let $H \subseteq R \subseteq S$ be subrings of $S$
    such that $H$ is a directed union of subrings $D_\alpha=B_\alpha \cap R_\alpha$, $\alpha \in {\cal A}$, where $B_\alpha$ and $R_\alpha$ are subrings of $S$ with $R = \bigcup_{\alpha}R_\alpha$.  Suppose that for all $\alpha,\beta \in {\cal A}$,  whenever $D_\alpha \subseteq D_\beta$, then $R_\alpha \subseteq R_\beta$.
            Then there exists an ultrafilter ${\cal U}$ on ${\cal A}$ such that, identifying each subring of $S$ with its image in $\prod_{\cal U}S$ under the diagonal mapping, we have that $H = (\prod_{\cal U}B_\alpha) \cap R$.
   \end{lem}

\begin{proof}  For each $\alpha \in {\cal A}$ define $U_\alpha=\{\beta \in {\cal A}:D_\alpha \subseteq D_\beta\}$.
First we claim that the collection $\{U_\alpha:\alpha \in {\cal
A}\}$ extends to  an ultrafilter ${\cal U}$ on ${\cal A}$.  To prove
this, it is enough to show that $\{U_\alpha:\alpha \in {\cal A}\}$
has  the finite intersection property (see for example the proof of
Proposition 3.3.6 in \cite{CK}). Let $\alpha_1,\ldots,\alpha_n \in
{\cal A}$.  Then since $H$ is a directed union of the $D_\alpha$,
there exists $\beta \in {\cal A}$ such that $D_{\alpha_1} + \cdots +
D_{\alpha_n} \subseteq D_\beta$, and hence  $\beta \in  U_{\alpha_1} \cap
\cdots \cap U_{\alpha_n}$. Therefore, $\{U_\alpha:\alpha \in {\cal
A}\}$ satisfies the finite intersection property  and  can be
extended to an ultrafilter ${\cal U}$ on ${\cal A}$.

Next we claim that $H \subseteq (\prod_{\cal U} B_\alpha) \cap R$.
(Recall our convention that all these rings are subrings of the ring
$\prod_{\cal{U}}S$.) Since by assumption $H \subseteq R$, we need
only check that $H \subseteq \prod_{\cal U} B_\alpha$.  Let $h \in
H$.  Then since $H = \bigcup_{\alpha}D_\alpha$, there
exists $\alpha \in {\cal A}$ such that $h \in D_\alpha$, and this
implies that  $h \in D_\beta  \subseteq B_\beta$ for all $\beta \in
U_\alpha$, and hence $U_\alpha \subseteq \{\beta \in {\cal A}:h \in
B_\beta\}$.  Since ${\cal U}$ is a filter and $U_\alpha \in {\cal
U}$, we have $\{\beta \in {\cal A}:h \in B_\beta\} \in {\cal U}$,
so that $h \in \prod_{\cal U}B_\alpha$.  This shows that $H
\subseteq (\prod_{\cal U}B_\alpha) \cap R$.

It remains to verify the reverse inclusion,  $(\prod_{\cal
U}B_\alpha) \cap R \subseteq H$.  Let $r \in (\prod_{\cal
U}B_\alpha) \cap R$.  Then since $R$ is a union of the $R_\alpha$'s,
there exists $\beta \in {\cal A}$ such that $r \in R_\beta$.  Also,
since $r \in \prod_{\cal U}B_\alpha$, we have
 $\{\alpha \in {\cal A}: r \in B_\alpha\} \in {\cal U}$.  Thus since ${\cal U}$ is a filter, $U_\beta \cap \{\alpha \in {\cal A}: r \in B_\alpha\} \in {\cal U}$.  Consequently, since an ultrafilter cannot contain the empty set, there exists $\gamma \in U_\beta$ such that $r \in B_\gamma$.
Since $\gamma \in U_\beta$, we have
$D_\beta \subseteq D_\gamma$.  Thus by assumption, $R_\beta \subseteq
R_\gamma$, so that $r \in B_\gamma \cap R_\gamma = D_\gamma
\subseteq H$.  Thus proves that $(\prod_{\cal U}B_\alpha) \cap R
\subseteq H$, which completes the proof.
\end{proof}

\begin{prop} \label{n case} Let $D \subseteq H \subseteq R$ be domains with $R$ a proper overring of $D$ and $H$ integrally closed, and let $n$ be a positive integer.    If $H$ is a directed union of overrings $A_\alpha$ of $D$  such that  each $A_\alpha$ is an essentially $n$-valuated subring of $R$, then $H$ is an essentially $n$-valuated subring of $R$.
\end{prop}

\begin{proof}  Let $P$ be a prime ideal of $H$, and define:
$$P_\alpha = A_\alpha \cap P \:\:\:\:\:\:\:\:\:\:\:\: D_\alpha =(A_\alpha)_{P_\alpha} \:\:\:\:\:\:\:\:\:\:\:\:  R_\alpha = R_{P_\alpha} = RD_\alpha.$$
 By assumption, for each $\alpha$, there exists an overring $B_\alpha$ of $D_\alpha$ such that $B_\alpha$ is an intersection of $n$ (not necessarily distinct) valuation overrings of $D_\alpha$ and
 $D_\alpha = B_\alpha \cap R_\alpha$.  We use Lemma~\ref{ultraproduct} to show that there exists an ultrafilter ${\cal U}$ on ${\cal A}$ such that $H_P = (\prod_{\cal U}B_\alpha) \cap R_P$.  First, note that $H_P$ is a directed union of the $D_\alpha$, and $R_P$ is a union of the $R_\alpha$.    Moreover, if $D_\alpha \subseteq D_\beta$, then $R_\alpha = R_{P_\alpha} = RD_\alpha \subseteq RD_\beta = R_\beta$.  Therefore, we may apply Lemma~\ref{ultraproduct} to obtain an ultrafilter ${\cal U}$ with $H_P = (\prod_{\cal U}B_\alpha) \cap R_P$.  Now by a theorem of Nagata,  each $B_\alpha$, since it is an intersection of $n$ valuation overrings, is a Pr\"ufer domain having at most $n$ maximal ideals \cite[(11.11), p.~38]{Nagata}.
 Thus the ultraproduct $\prod_{\cal U}B_\alpha$ is a Pr\"ufer domain having at most $n$ maximal ideals; see
 \cite[Propositions 2.2 and 4.1]{OS}.  Consequently, this ultraproduct is an intersection of $n$ not necessarily distinct valuation overrings $W_1,\ldots,W_n$. Let $F$ denote the quotient field of $H$, and for each $i=1,2,\ldots,n$, let  $V_i = W_i \cap F$.  Then $H_P = (\prod_{\cal U}B_\alpha) \cap R_P = V_1 \cap V_2 \cap \cdots \cap V_n \cap R_P$.  Since each $W_i$ is a valuation ring,  each $V_i$ is a valuation ring, and this proves the lemma.
 \end{proof}

The next proposition interprets the property of the ring $H$  being essentially $n$-valuated in terms of height $1$ prime ideals of $H$ lying over the maximal ideal of $D$. This connection will become more transparent in the next section when we restrict to overrings of two-dimensional Noetherian domains.

\begin{prop} \label{representation2}  Let $D \subseteq H \subseteq R$ be domains with $R$ a proper overring of $D$ and $H$ integrally closed, and suppose that  there exist  prime ideals $P_1,P_2,\ldots,P_t$ of $H$ with $t>0$ such that $H = H_{P_1} \cap H_{P_2} \cap  \cdots \cap H_{P_t} \cap R$ and each $H_{P_i}$ is a DVR that is irredundant in this representation.  Let  $n$ be a positive integer. If  $t \leq n$, then $H$ is an essentially $n$-valuated subring of $R$.  Otherwise, suppose $n < t$.   Then:
\begin{itemize}
\item[(1)]  $H$ is an essentially $n$-valuated subring of $R$ if and only if $$H = P_{i_i} + P_{i_2} + \cdots + P_{i_{n+1}}$$ whenever $i_1,i_2,\ldots,i_{n+1}$ are distinct members of $\{1,2,\ldots,t\}$.

\item[(2)]
 $H$ is an essentially one-valuated subring of $R$ if and only if the prime ideals $P_1,P_2,\ldots,P_t$ are pairwise comaximal.
\end{itemize}
\end{prop}

\begin{proof} If $t \leq n$, then since each $H_{P_i}$ is a DVR, and the definition of essentially $n$-valuated subrings does not require the valuation rings in the representation to be distinct, it is the case that $H$ is an essentially $n$-valuated subring of $R$. So suppose for the rest of the proof that $n < t$.  
Statement (2) is immediate from (1), so we need only prove (1).  Suppose that $H = P_{i_i} + P_{i_2} +
\cdots + P_{i_{n+1}}$ whenever $i_1,i_2,\ldots,i_{n+1}$ are distinct
members of $\{1,2,\ldots,t\}$. Let $M$  be  a maximal ideal of $H$.  Then at most  $n$ members of
$\{P_1,P_2,\ldots,P_t\}$ can be contained in $M$.
 Now $H_M = (H_{P_1})_M \cap
\cdots \cap (H_{P_t})_M \cap R_M$. Also, $(H_{P_i})_M = H_{P_i}$ if
and only if $P_i  \subseteq M$.  For these prime ideals $P_i$
contained in $M$, we have $(H_{P_i})_M = H_{P_i}$, while for the
prime ideals $P_j \not \subseteq M$, it must be that $H_{P_j}$ is a
proper subring of $(H_{P_j})_M$, and hence, since $H_{P_j}$ is a
DVR, $(H_{P_j})_M$ is the quotient field of $D$. Therefore, $H_M$
can be represented using $R_M$ and fewer than $n+1$ of the DVRs
$H_{P_i}$.  This shows that $H$ is an essentially $n$-valuated subring of
$R$.

 Conversely, suppose that $H$ is an essentially $n$-valuated subring of $R$.
By way of contradiction, suppose without loss of generality that
  $P_1 + P_2+\cdots + P_{n+1}$ is contained in a maximal ideal $M$ of $H$.
  By assumption we may write
  $$H_M = W_1 \cap W_2 \cap \cdots \cap W_n \cap R_M,$$ where the $W_i$'s are not necessarily distinct valuation overrings of $H_M$.
  Therefore, for each $i=1,2,\ldots,n+1$, since $P_i \subseteq M$, we may localize this representation of $H_M$ to obtain
   $$H_{P_i} = (W_1)_{P_i} \cap (W_2)_{P_i} \cap \cdots \cap (W_{n})_{P_i} \cap R_{P_i}.$$ Since $H_{P_i}$ is a DVR, its only overrings are itself and its quotient field, so one of the rings in the intersection must be equal to $H_{P_i}$. Thus since $R \not \subseteq H_{P_i}$,
   it follows that $H_{P_i} = (W_j)_{P_i}$ for some $j$, which forces
    $W_j\subseteq H_{P_i}$.
    Thus for each $i =1,\ldots,n+1$, there exists $j=1,\ldots,n$, such that $W_j \subseteq H_{P_i}$. Consequently, one of the $W_j$ must be contained in $H_{P_i}$ for two choices of $i$.
   After relabeling, we may assume that $W_1$ is contained in both $H_{P_1}$ and $H_{P_2}$.  However, $H_{P_1}$ and $H_{P_2}$ are DVR overrings of the valuation ring $W_1$, and since the overrings of a valuation ring are totally ordered with respect to inclusion, this forces one of $H_{P_1}, H_{P_2}$ to be contained in the other. But since, as noted above, the only valuation overrings of a DVR are itself and its quotient field, we conclude that $H_{P_1} = H_{P_2}$, contrary to the assumption that no $H_{P_i}$ can be omitted from the representation $H = H_{P_1} \cap H_{P_2} \cap \cdots \cap H_{P_t} \cap R$.  Therefore, $H = P_{i_i}  + \cdots + P_{i_{n+1}}$whenever $i_1,i_2,\ldots,i_{n+1}$ are distinct members of $\{1,2,\ldots,t\}$.
   \end{proof}

\section{Essentially $n$-fibered extensions}

In this section we show how the $n$-fibered and $n$-valuative properties are related.  Specifically, we show in Theorem~\ref{characterization} that when $D$ is an integrally closed local Noetherian domain of dimension $2$ and $f \in D$ is such that $\sqrt{fD}$ is a prime ideal, then every ring between $D$ and $D_f$ is an essentially $n$-valuated subring if and only if every finitely generated $D$-subalgebra of $D_f$ is essentially $n$-fibered.  A first step in this direction is the following technical  lemma regarding the decomposition of rings between $D$ and $D_f$.

\begin{lem} \label{representation}
Let $D$ be a two-dimensional integrally closed local
Noetherian domain,  let ${\ff m}$ be the
maximal
 ideal of $D$, and let $0 \ne f,g \in {\ff m}$.  Assume
  \begin{itemize}

  \item[(i)]
   $H$  is a two-dimensional integrally closed Noetherian overring of $D$ that is contained in $D_f;$

 \item[(ii)]  each prime ideal in $\Ass(f)$  is the radical of a principal ideal;

   \item[(iii)] $g \in \{h \in D:H \subseteq D_h {\mbox{ and }} \Ass(h) \subseteq \Ass(f)\},$ and for any other element $h$ in this set, $|\Ass(g)| \leq |\Ass(h)|$; and

   \item[(iv)]
  $P_1,\ldots,P_t$ are the height $1$ prime ideals of $H$ that contain $g$.

  \end{itemize}
Then $P_1,\ldots,P_t$ are the exceptional prime ideals of the extension $H/D$; $H$ may be represented as  $H =  H_{P_1} \cap \cdots \cap H_{P_t} \cap D_g;$ and no ring in this representation can be omitted. If also $X:=\Ass(f) \setminus \Ass(g)$ is nonempty, then $$(\dagger) \:\:\:\:\:\:\:\: H =
H_{P_1} \cap H_{P_2} \cap \cdots \cap H_{P_t} \cap D_f \cap \left(\bigcap_{{\ff p} \in X} D_{{\ff p}}\right),$$
where also no ring in this representation can be omitted.
\end{lem}

\begin{proof}  Since $D$ is an integrally closed local Noetherian domain of Krull dimension $2$, $D$ is a Cohen-Macaulay ring, and hence associated prime ideals of  nonzero proper principal ideals have height $1$.   In particular,  $\Ass(f)$ and $\Ass(g)$ consist of height one prime ideals of $D$.
Moreover, note that (iv) is not vacuous: There is at least one height $1$ prime ideal containing $g$,
since otherwise $g$ is a unit in $H$,
which since $H \subseteq D_g$, forces $H=D_g$, and  this contradicts the fact that $H$ has dimension $2$.

First we claim that there exists $h \in D$ such that $h/g \in H$ and $(h,g)$ is an ${\ff m}$-primary ideal of $D$.  To this end, we show first that $gH \cap D$ is an ${\ff m}$-primary ideal of $D$.  It must be that  $gH \cap D \ne D$, since otherwise (iii) implies that $H = D_g$, contrary to (i), where $H$ is assumed to have dimension $2$.  Suppose by way of contradiction that $gH \cap D$ is not an ${\ff m}$-primary ideal of $D$.  Since $gH \cap D$ is a proper ideal of $D$ and $D$ is local,
 there must then exist a height $1$ prime ideal ${\ff p}$ of $D$ such that $gH \cap D \subseteq {\ff p}$.   Thus $g \in {\ff p}$, so that by (iii), ${\ff p} \in \Ass(g) \subseteq \Ass(f)$.
Localizing at ${\ff p}$ produces $gH_{\ff p} \cap D_{\ff p} \subseteq {\ff p}D_{\ff p}$.
 Yet since $D_{\ff p}$ is a valuation ring, its fractional ideals are linearly ordered by inclusion, so necessarily $gHD_{\ff p} \subseteq {\ff p} D_{\ff p}$.  In particular, $HD_{\ff p}$ is not the quotient field of $D$, and hence since $D_{\ff p}$ is a DVR (so in particular its only overrings
are itself and its quotient field), we have that $H \subseteq D_{\ff p}$.  Therefore, $H \subseteq D_g \cap D_{\ff p}$.  Since by (ii) and (iii), each prime ideal minimal over $g$ is the radical of a principal ideal, we may choose $h \in D$ such that $\Ass(h) \subseteq \Ass(g)$ and $h$ is in every prime ideal of $D$ containing $g$ except ${\ff p}$. In particular, the  height one prime ideals of $D$ {not} containing $h$ are precisely ${\ff p}$ and  the height  one prime ideals that do not  contain $g$. It follows that $D_h = D_g \cap D_{\ff p}$. (Indeed,
 since $D$ is a Krull domain, then for each $0 \ne a \in D$, the ring $D_a$ is the intersection of the rings $D_{\ff q}$, where ${\ff q}$ ranges over the  height $1$ prime ideals of $D$ that do not contain $a$.)
 But then $H \subseteq D_g \cap D_{\ff p} = D_h$
with $\Ass(h) \subseteq \Ass(g) \subseteq \Ass(f)$ and $|\Ass(h)| < |\Ass(g)|$, which contradicts  (iii).  Therefore, $gH \cap D$ must be an ${\ff m}$-primary ideal of $D$, and since $D/gD$ is a one-dimensional local ring we may choose $h \in gH \cap D$ such that $(g,h)$ is an ${\ff m}$-primary ideal of $D$.  Moreover, since $h \in gH$, we have $h/g \in H$, and this proves the claim that there exists $h \in D$ such that $h/g \in H$ and $(h,g)$ is an ${\ff m}$-primary ideal of $D$.

Next, since $h/g \in H$, it follows that $gH =
(h,g)H$, so that any prime ideal of $H$  containing $g$
contains $(h,g)D$, and hence, since this ideal is ${\ff m}$-primary,
contains also ${\ff m}$. Therefore, each $P_i$ must contain ${\ff m}$, and it follows that $\{P_1,\ldots,P_t\}$ is the set of
 exceptional prime ideals of $H/D$.
Now since $H$ is a
Krull domain, we have $$H = \left(\bigcap_{g \in Q}H_Q \right) \cap
\left( \bigcap_{g \not \in Q} H_Q \right),$$ where $Q$ ranges over
the height $1$ prime ideals of $H$.  Also, since $D \subseteq H
\subseteq D_g$, we have $D_g = H_g$.  Therefore, since a height one prime ideal $Q$ of $H$ contains $g$ if and only if $Q = P_i$ for some $i$, the fact that $H$ is a Krull domain (so that $H$ is the intersection of $H_g$ and the localizations of $H$  at the height one prime ideals of $H$ that do not contain $g$) implies that
$H = H_{P_1} \cap \cdots \cap H_{P_t} \cap D_g.$  Since the decomposition of a Krull domain in terms of localizations at its height one prime ideals is irredundant, no ring can be omitted from this representation of $H$.  Furthermore, since the  set of height one prime ideals of $D$ not containing $g$ is precisely the union of $X$ with the set of height one prime ideals of $D$ not containing $f$, we use the fact again that $D$ is a Krull domain to obtain
 $D_g = D_f \cap (\bigcap_{{\ff p} \in X} D_{\ff p})$.
  The second representation of $H$ given by $(\dagger)$
     follows.  This representation, since it also arises from the decomposition of a Krull domain, is  irredundant.
  \end{proof}




In the case where $f$ is a prime element of $D$, or more generally, has prime radical, the proposition simplifies to assert that $H =  H_{P_1} \cap \cdots \cap H_{P_t} \cap D_f,$ where  $\{P_1,\ldots,P_t\}$ is the set of  exceptional prime ideals of the extension $H/D$.

Combining Proposition~\ref{representation2} and Lemma~\ref{representation}, we make the connection between the $n$-fibered and $n$-valuated properties more explicit:

\begin{lem}  \label{connection1}
With the same assumptions as
Lemma~\ref{representation}, let
 $k = |\Ass(f) \setminus \Ass(g)|$ and $n$ be a positive integer.
 \begin{itemize}
 \item[(1)] If  $H$ is an essentially $n$-fibered extension of $D$, then ${H}$ is an essentially $(n+k)$-valuated subring of $D_f$.

  \item[(2)]
  If $H$ is an essentially $n$-valuated subring of $D_f$, then $H$ is an essentially $n$-fibered extension of $D$.\end{itemize}\end{lem}

\begin{proof}
Using Lemma~\ref{representation}, write $H$ as in $(\dagger)$, where
if $k =0$, we omit the last component $\bigcap_{{\ff p} \in X}D_{\ff p}$ of the representation.

(1) Suppose  that the ring $H$ is an essentially $n$-fibered extension of $D$.  Then for each maximal ideal $M$ of $H$, the extension $H_M/D$ has at most $n$ exceptional prime ideals.  Consequently, the sum of any $n+1$ exceptional prime ideals of the extension $H/D$ must be equal to $H$, and so by Proposition~\ref{representation2}, $H$ is an essentially $n$-valuated subring of $D_f \cap (\bigcap_{{\ff p} \in X}D_{\ff p}).$  Thus since $X$ has $k$ elements, $H$ is an essentially $(n+k)$-valuated subring of $D_f$.

(2) Suppose that the ring  $H$ is an essentially $n$-valuated subring of $D_f$.
We claim that the extension $H/D$ is an essentially $n$-fibered extension.
     Write $X = \{ {\ff q}_1,\ldots,{\ff q}_k\}$, and for each $i =1,2,\ldots,k$, let $Q_i = {\ff q}_iD_{{\ff q}_i} \cap H.$  Then since $D_{{\ff q}_i}$ is  a DVR, it follows that for each $i$, $H_{Q_i} = D_{{\ff q}_i}$.  Therefore, we have from $(\dagger)$ that   $$H = H_{P_1} \cap \cdots \cap H_{P_t} \cap H_{Q_1} \cap \cdots \cap H_{Q_k} \cap D_f.$$ The set of height $1$ prime ideals of $H$  containing $f$ is $\{P_1,\ldots,P_t,Q_1,\ldots,Q_k\}$, and by Proposition~\ref{representation2}, any $n+1$ prime ideals in $\{P_1,\ldots,P_t,Q_1,\ldots,Q_k\}$ must sum to $H$.
  Let $M$ be a maximal ideal of $H$ containing ${\ff m}$.  Then
  since any height $1$ prime ideal of $H$ containing the maximal ideal of $D$ must contain $f$,  it follows that the exceptional prime ideals of the extension $H_M/D$ are among $\{P_1,\ldots,P_t,Q_1,\ldots,Q_k\}$, and hence since at most $n$ of these prime ideals are contained in $M$, we conclude that the extension $H_M/D$ has at most $n$ exceptional prime ideals.   Therefore, $H/D$ is an essentially $n$-fibered extension.
\end{proof}

When, as in the lemma,  $fD$ has a prime radical,
then the next theorem shows that  whether the integrally closed rings $H$ with $D \subseteq H
\subseteq D_f$ are $n$-valuated is determined by the finitely
generated $D$-subalgebras of $D_f$, and in particular,  by the rings
that can be represented as an intersection of $D_f$ and finitely many
hidden prime divisors of $D$.

\begin{lem}  \label{very new lemma}
 Let $D$ be a two-dimensional integrally closed local
 Noetherian domain, and suppose $f$ is an element of $D$ such that  $\sqrt{fD}$ is a prime ideal. Then $D_f$ is the only one-dimensional  finitely generated $D$-subalgebra of $D_f$. 
 \end{lem}

\begin{proof} 
Let $H$ be a  one-dimensional finitely generated $D$-subalgebra  of $D_f$.    
Then by 
Zariski's Main Theorem, no prime ideal of $H$ lies over  the maximal ideal of $D$ \cite[Theorem, p.~45]{Evans}. Suppose by way of contradiction that  ${H} \subsetneq D_f$. Then there is a prime ideal $P$ of ${H}$ such that $P \cap D$ contains $f$ (since otherwise, $D_f \subseteq H$), and hence $P \cap D$ is  a height one prime ideal of $D$ containing $f$, which since $\sqrt{fD}$ is a prime ideal, implies $P \cap D = \sqrt{fD}$.  But also
 $D_{P \cap D}$ is a DVR that is contained in $H_P$, which, since the only overrings of a DVR are itself and its quotient field, forces $D_{P \cap D} = H_P$. However, this implies that $H \subseteq D_{P \cap D} \cap D_f = D$, and hence $H = D$, a contradiction to the assumption that $H$ has dimension one. Consequently,  $H = D_f$. 
 \end{proof}

 \begin{thm} \label{characterization}
 Let $D$ be a two-dimensional integrally closed local
 Noetherian domain.  Suppose that $n$ is a positive integer and $f$ is an element of $D$ such that  $\sqrt{fD}$ is a prime ideal. Then the following statements are equivalent.

\begin{itemize}

\item[(1)]  Every $D$-subalgebra of $D_f$ is essentially $n$-valuated.

\item[(2)]  Every finitely generated $D$-subalgebra of $D_f$ is essentially $n$-valuated.


\item[(3)]  Every finitely generated $D$-subalgebra  of $D_f$ is essentially $n$-fibered.



\end{itemize}
\end{thm}

\begin{proof} (1) $\Rightarrow$ (2) This is clear.

(2) $\Rightarrow$  (3) Let $H$ be the integral closure of a finitely generated $D$-subalgebra of $D_f$. If $H$ has dimension one, then by Lemma~\ref{very new lemma}, $H = D_f$, 
%
and hence $H$ is vacuously  essentially $n$-fibered. On the other  hand, if  $H$ has Krull  dimension $2$, then by Lemma~\ref{connection1} (since $k=0$ in the present case), $H$ is essentially $n$-fibered.

(3) $\Rightarrow$ (2)  Let $H$ be the integral closure of a finitely generated $D$-subalgebra of $D_f$. If $H$ has Krull dimension $1$, then as above, $H = D_f$, and (2) holds  trivially. Otherwise, if $H$ has dimension $2$, then by Lemma~\ref{connection1}, with again $k=0$, we have that $H$ is essentially $n$-valuated.

(2) $\Rightarrow$ (1)  Let $H$ be an integrally closed domain with
$D \subseteq H \subseteq D_f$.  Let $\{F_\alpha:\alpha \in {\cal
A}\}$ be the collection of all finite subsets of $H$, and  for each
$\alpha$, let $D_\alpha=D[F_\alpha]$.   Then $H$ is a directed union of the
$D_\alpha$, and by (2), each $D_\alpha$ is an essentially $n$-valuated
subring of $D_f$, so that by Proposition~\ref{n case}, $H$ is an essentially
$n$-valuated subring of $D_f$.
\end{proof}

\begin{rem} {\em
Although we do not make use of it, we mention here that the choice of valuations in the local representation of $H$ is unique, when uniqueness is formulated carefully.  In general, let $H \subseteq R$ be integrally closed overrings of a two-dimensional Noetherian domain $D$ such that $H = V_1 \cap \cdots \cap V_n \cap R = W_1 \cap \cdots \cap W_m \cap R$, where the $V_i$ and $W_j$ are valuation overrings of $H$.  After replacing some of the $V_i$ and $W_j$ with proper overrings where necessary, we may assume that the  $V_i$ and $W_j$ are {\it strongly irredundant} in their respective representations, meaning no $V_i$ or $W_j$ can be replaced in the representation by a proper overring.    Then $\{V_1,\ldots,V_n\} = \{W_1,\ldots,W_m\}$ \cite[Corollary 5.6]{OlbI}.  More general uniqueness results can be found in \cite{OlbI} and \cite{OlbS}.}
\end{rem}

%

Using Theorem~\ref{characterization}, we now prove Theorem~\ref{first theorem}.
\smallskip

{\it Proof of Theorem~\ref{first theorem}.} The equivalence of (2) and (3) when $fD$ has prime radical follows from Theorem~\ref{characterization}. Thus it remains to prove the equivalence of (1) and (2).
Before doing so, we first observe that (with $D$ as in the theorem), if $\pi:X \rightarrow \Spec(D)$ is a modification with normalization $\overline{\pi}:\overline{X} \rightarrow \Spec(D)$, then the induced morphism $\overline{\pi}_f:\overline{X}_f \rightarrow \Spec(D_f)$ is an isomorphism of schemes, and hence there is an open immersion $\Spec(D_f) \rightarrow \overline{X}$. For
 if ${\ff p} \in \Spec(D)$ with $f \not \in {\ff p}$, then necessarily $D$ has height $1$, so that $D_{\ff p}$ is a DVR. Thus by the Valuative Criterion for Properness, $D_{\ff p}$ must dominate the local ring of a point in $X$, hence also a point $x$ in $\overline{X}$ which maps under $\overline{\pi}$ to ${\ff p}$.  It follows that ${\cal O}_{\overline{X},x} = D_{\ff p}$.
Moreover, if $y \in \overline{X}_f$, then since $f$ is in the maximal ideal of $D$ and $D$ has dimension $2$, then ${\cal O}_{\overline{X},y}$ dominates $D_{\ff q}$ for ${\ff q}$ a height one prime ideal of $D$, so that since $D_{\ff q}$ is a DVR and ${\cal O}_{\overline{X},y}$ is an overring of $D_{\ff q}$, we conclude ${\cal O}_{\overline{X},y} = D_{\ff q}$.
  These considerations shows that $\overline{\pi}_f$ is an isomorphism of schemes.

Now to see that (1) implies (2),
 let $H$ be the integral closure of a finitely generated $D$-subalgebra of $D_f$. Then there exists an ideal $I$
 of $D$ and $k>0$ such that $f^k \in I$ and $H = \overline{D[If^{-k}]}$.  Let $u$ be an indeterminate for $D$, and let $X=$ Proj$(D[Iu])$, so that
the induced morphism $\pi:X \rightarrow \Spec(D)$ is the blow-up of $\Spec(D)$ along the ideal $I$.
 The morphism $\pi$ is clearly birational; also, as a  projective morphism, $\pi$ is proper. Hence $\pi$ is a modification of $\Spec(D)$. Moreover, $\Spec(H)$ is an affine open subscheme  of $\overline{X}$.
Since $H \subseteq D_f$, then the induced morphism $\Spec(D_f) \rightarrow \Spec(H) \subseteq \overline{X}$ is an open immersion, and since the map $\overline{\pi}_f$ defined above is an isomorphism of schemes, it follows that $\overline{X}_f   \subseteq \Spec({H}) \subseteq \overline{X}$.
Finally, let ${\ff p}_1,\ldots,{\ff p}_t$ be the height one prime ideals of $H$ lying over ${\ff m}$.  For each $i$, let $C_i$ be the Zariski closure of $\{ {\ff p}_i\}$ in $\Spec(H)$.
    Then by (1), at most $n$ of the $C_i$ meet in a point in $\Spec(H)$, and hence at most $n$ of the prime ideals ${\ff p}_1,\ldots,{\ff p}_t$ are contained in a maximal ideal of $H$.

Finally, to prove that (2) implies (1),  let  $\pi:X \rightarrow \Spec(D)$ be a modification, and let $U$ be an affine open subscheme of $\overline{X}$ containing $\overline{X}_f$.
Since $X$ is the normalization of a scheme of finite type over $\Spec(D)$, there exists a finitely generated $D$-subalgebra
 of the quotient field of $D$ whose normalization $H$ is such that $U = \Spec(H)$. Moreover, since $\overline{X}_f \subseteq U$, then $H = {\cal O}_{\overline{X}}(U) \subseteq {\cal O}_{\overline{X}}(\overline{X}_f) \subseteq  D_f$. (This last containment follows from the
 fact that the map $\overline{\pi}_f:\overline{X}_f \rightarrow \Spec(D_f)$ defined above is an isomorphism of schemes.)
Now let $C_1, \ldots, C_t$ be the irreducible components of $\overline{\pi}^{-1}({\ff m})$, and let $\xi_1,\ldots,\xi_t$ be the corresponding generic points of these components. Suppose that $t \geq n+1$ and there exists a point $x \in U$ that is in the intersection of $n+1$ of these components. Without loss of generality, $x \in C_1 \cap \cdots \cap C_{n+1} \cap U$, and  since $U$ is an open set, then $\xi_1,\ldots,\xi_{n+1}  \in U$.  Let ${\ff p}_1,\ldots,{\ff p}_{n+1}$ be the prime ideals of $H$ corresponding to $\xi_1,\ldots,\xi_{n+1}$, respectively, and let ${\ff q}$ be the prime ideal corresponding to $x$. Then since $x$ is in the closure of each $\xi_i$, it follows that ${\ff p}_i \subseteq {\ff q}$ for each $i=1,\ldots,n+1$.  As the normalization of a finitely generated algebra over $D$, $H$ has Krull dimension at most $2$, so each ${\ff p}_i$ is a height one prime ideal of $H$ lying over ${\ff m}$, while ${\ff q}$ is a maximal ideal of $H$ containing each ${\ff p}_i$. But this is impossible by (2), so at most $n$
irreducible components  of  $\overline{\pi}^{-1}({\ff m})$ meet in $U$.  This proves Theorem~\ref{first theorem}. $\:\:\:\:\:\square$

\section{Passage to the completion}

In this section we consider how the  $n$-valuated property
transfers to the completion $\widehat{D}$ of $D$ in the ${\ff
m}$-adic topology, where ${\ff m}$ is the maximal ideal of the local ring $D$. We
see in the next section that when $D$ is a two-dimensional regular
local ring, then passage to the completion $\widehat{D}$ simplifies
things, and allows us to apply MacLane's construction of hidden
prime divisors via key polynomials.

But first we need a lemma that shows there is a bijection between
the hidden prime divisors of $D$ and those of $\widehat{D}$, when
$D$ is analytically irreducible (that is, when $\widehat{D}$ is a
domain).  More precisely, if $D$ is an analytically irreducible
local Noetherian domain with maximal ideal ${\ff m}$ and Krull
dimension $d>1$, then there is a one-to-one correspondence between
hidden prime divisors of $D$ and those of $\widehat{D}$.  Part of
this claim is proved by Abhyankar in \cite[pp.~513--514]{A}, and the
remaining details  are included by G\"ohner in \cite[Lemma
1.10]{Go}.   A self-contained presentation of the lemma can be found in \cite[Proposition 9.3.5]{SH}.


\begin{lem} \label{Abhyankar} {\em (Abyhankar, G\"ohner)} Let $D$ be an analytically irreducible local Noetherian domain of Krull dimension $d>1$.   Then there is a one-to-one correspondence between the hidden prime divisors $V$ of $D$ and the hidden prime divisors $W$ of  $\widehat{D}$ given by $V \mapsto V'$, where   $V'$ is any extension of $V$ to a valuation overring of $\widehat{D}$, and $W \mapsto W \cap F$, where $F$ is the quotient field of $D$.  Under this correspondence $V$ and $V'$ have the same value group and residue field.
\end{lem}

The next lemma is an application of the preceding lemma and the
approximation theorem for valuations.

\begin{lem} \label{approximation}  Let $D$ be an analytically irreducible local Noetherian domain of Krull dimension $>1$, and let $v_1,v_2,\ldots,v_n$ be valuations corresponding to hidden prime divisors of $\widehat{D}$.  Then for each $z \in \widehat{D}$ there exists $d \in D$ such that $v_i(z) = v_i(d)$ for every $i=1,2,\ldots,n$.
\end{lem}

\begin{proof} Let $z \in \widehat{D}$, and for each $i$, let $w_i$ be the restriction of $v_i$ to the quotient field of $D$.  Then
by Lemma~\ref{Abhyankar}, for each $i$, the value groups of $w_i$ and $v_i$ are the same, so there exists $d_{i} \in D$ such that $v_i(z) = v_i(d_{i})$. Therefore, by the approximation theorem applied to the collection of valuation rings $\{w_1,w_2,\ldots,w_n\}$, there exists $d \in D$ such that for each $i$, $w_i(d_i)= w_i(d)$ \cite[Theorem 2.4.1, p.~48]{EP}.  Hence for each $i$, $v_i(z) = v_i(d_i) = w_i(d_i) = w_i(d) = v_i(d)$.
\end{proof}

\begin{lem}  \label{irredundant}
Let $D$ be an analytically irreducible   local Noetherian domain of
Krull dimension $>1$, let $V_1,\ldots,V_n$ be hidden prime divisors
of $D$, and let $V_1',\ldots,V'_n$ be  extensions of these valuation
rings to valuation overrings of $\widehat{D}$.  Let $0 \ne f \in D$.
Then for each $i$, $V_i$ is irredundant in the intersection $V_1
\cap \cdots \cap V_n \cap D_f$ if and only if $V'_i$ is irredundant
in the intersection $V'_1 \cap \cdots \cap V'_n \cap \widehat{D}_f$.
\end{lem}

\begin{proof}
   Let $i \in \{1,2,\ldots,n\}$.
    If $(\bigcap_{j \ne i} V'_j) \cap \widehat{D}_f \subseteq V'_i$, then intersecting both sides of the containment with the quotient field $F$, we obtain since each $V'_j$ extends $V_j$ that  $(\bigcap_{j \ne i}V_j) \cap D_f \subseteq V_i$.  Conversely, suppose that $V'_i$ is irredundant in the intersection    $V'_1 \cap \cdots \cap V'_n \cap \widehat{D}_f$.
    Then there exists $d \in \widehat{D}$ and $s>0$ such that $d/f^s \in (\bigcap_{j \ne i}V'_j)  \setminus V'_i$.  Moreover, by Lemma~\ref{approximation} we may assume   that $d \in D$.
      Therefore, $d/f^s \in (\bigcap_{j \ne i}V_j)  \setminus V_i$, so that $V_i$ is irredundant in the representation $V_1 \cap V_2 \cap \cdots \cap V_n \cap D_f$.      \end{proof}

In passing to the completion $\widehat{D}$ of $D$, we wish to apply results from the previous sections, and so we need that not only $D$, but also $\widehat{D}$ is normal, i.e., that $D$ is {\it analytically normal}.  In the next section our main application  of the following theorem is to the case where $D$ is a regular local ring.  In this case, since $\widehat{D}$ is also a regular local ring, analytic normality of $D$ is immediate since regularity is preserved by completion.

\begin{thm} \label{pass to completion} Let $D$ be a two-dimensional  analytically normal
local Noetherian domain,  let $0 \ne f \in D$ such that  the radical of $f\widehat{D}$ is a prime ideal, and let $n$ be a
positive integer. Then the following statements are equivalent.

 \begin{itemize}

 \item[(1)]  Every $D$-subalgebra of $D_f$
 is  essentially $n$-valuated.

\item[(2)]  Every $\widehat{D}$-subalgebra of   $\widehat{D}_f$ is  essentially $n$-valuated.

 \item[(3)]  Every finitely generated $D$-subalgebra of $D_f$ is an essentially $n$-fibered extension of $D$.

 \item[(4)]  Every finitely generated $\widehat{D}$-subalgebra of $\widehat{D}_f$ is an essentially $n$-fibered extension of $\widehat{D}$.

\end{itemize}

\end{thm}

\begin{proof}  Since $f\widehat{D}$ has prime radical, so does $fD$.  For suppose that ${\ff p}_1$ and ${\ff p}_2$ are height 1 prime ideals of $D$ containing $f$.
Then since $\widehat{D}$ is faithfully flat over $D$, there are prime ideals ${\ff P}_1$ and ${\ff P}_2$ of $\widehat{D}$ minimal with respect to the property of lying over ${\ff p}_1$ and ${\ff p}_2$, respectively \cite[Theorem 7.3, p.~48]{Mat}.
Moreover, since $\widehat{D}$ is flat over $D$, then  ht(${\ff P}_i) = $ ht$({\ff p}_i) + \dim \widehat{D}_{{\ff P}_i}/{\ff p}_i\widehat{D}_{{\ff P}_i} = 1$
\cite[Theorem 15.1, p.~116]{Mat}. Thus  since the radical of  ${f\widehat{D}}$ is a prime ideal, it follows that ${\ff P}_1 = {\ff P}_2$, and hence ${\ff p}_1 = {\ff p}_2$.  Therefore, $fD$ has prime radical,  and the equivalence of (1) and (3), as well as the equivalence of (2) and (4), is given by Theorem~\ref{characterization}.

(1) $\Rightarrow$ (2)
      Suppose that every ring between $D$ and $D_f$ is an essentially   $n$-valuated subring of $D_f$.  We use Theorem~\ref{characterization} to verify that (2) holds.  Specifically,
      to show that (2) holds it is enough to prove that every finitely generated $\widehat{D}$-subalgebra of $\widehat{D}_f$ is essentially $n$-valuated.  Let $H'$ be the integral closure of a   finitely generated $\widehat{D}$-subalgebra of $\widehat{D}_f$.  Then by Lemma~\ref{representation},
      there are hidden prime divisors $V'_1,V'_2,\ldots,V'_t$  of $\widehat{D}$ such that  $H' = V'_1 \cap V'_2 \cap \cdots \cap V'_t \cap \widehat{D}_f$.
        We may assume that no $V'_i$ can be omitted from this representation.  For each $i$, let $V_i = V'_i \cap F$, where $F$ is the quotient field of $D$.
    Then by Lemma~\ref{irredundant} no $V_i$ can be omitted from the representation of $H := H' \cap F$ given by
      $H = V_1 \cap V_2 \cap \cdots \cap V_t \cap D_f$.  Thus by \cite[Lemma 1.3]{HO}, the fact that each $V_i$ is a DVR implies that for each $i$, $V_i = H_{P_i}$ for some height $1$ prime ideal $P_i$ of $H$.
        For each $i$, let $P'_i = {\ff M}_{V'_i} \cap H'$, where ${\ff M}_{V'_i}$ is the maximal ideal of $V'_i$.  Then, again by \cite[Lemma 1.3]{HO}, the fact that $V'_i$ is a DVR implies that $V'_i = H'_{P'_i}$, and in proving (2), without loss of generality it suffices by Proposition~\ref{representation2} to assume that $t>n$ and show that $P'_1 + P'_2 + \cdots P'_{n+1} = H'$.  In fact, since for each $i$, $P_i \subseteq P'_i$, it suffices to show that $P_1 + P_2 + \cdots + P_{n+1} = H$.  But this is the case by (1) and Proposition~\ref{representation2}, so
        the claim is proved.

(2) $\Rightarrow$ (3) Suppose that every  ring
between $\widehat{D}$ and $\widehat{D}_f$ is an essentially $n$-valuated
subring of $\widehat{D}_f$.
To prove (3), it suffices by
Theorem~\ref{characterization} to show that every finitely generated $D$-subalgebra of $D_f$ is an essentially
$n$-valuated subring of $D_f$. 
 Let $H$  be the integral closure of
a finitely generated $D$-subalgebra of $D_f$ in its quotient field.
If $H$ has dimension one, then by Lemma~\ref{very new lemma}, $H= D_f$, and hence $H$ is trivially an essentially $n$-valuated subring of $D_f$. So suppose that $H$ has dimension $2$.
Let $P_1,\ldots,P_t$ be the height $1$ prime ideals of the
Noetherian domain $H$ that contain ${\ff m}$, and for each
$i=1,\ldots,t$, let $V_i = H_{P_i}$.   Then by
Lemma~\ref{representation}, $H = V_1 \cap \cdots \cap V_t \cap
D_f$, and to prove that $H$ is an essentially $n$-valuated subring of
$D_f$,  we may assume that $t>n$ (or else there is nothing to show)
and by this same lemma
 we need only
verify that  for distinct elements $i_1,i_2,\ldots,i_{n+1}$ of
$\{1,2,\ldots,n\}$, it is the case that $P_{i_1} + P_{i_2}+\cdots + P_{i_{n+1}} = H$.
Without loss of generality it suffices to prove that $P_1 + P_2 +
\cdots + P_{n+1} = H$.  To this end, suppose by way of contradiction
that $P_1 + P_2 + \cdots + P_{n+1}$  is contained in a maximal ideal
$M$ of $H$. Since $H$ is a Noetherian domain, there exists a hidden
prime divisor $W$ of $H$ such that ${\ff M}_W \cap H = M$.  (For
example, consider the integral closure $E$ of the Noetherian ring
$H[M/h]$, where $0 \ne h \in M$, and choose a height $1$ prime ideal
$Q$ of $E$ containing $hE = ME$. Then $W:=E_{Q}$ is a hidden prime
divisor of $H$.)

Let $F$ denote the quotient field of $D$.  By Lemma~\ref{Abhyankar}
there exist unique hidden prime divisors $W',V'_1,V'_2,\ldots,V'_t$
of $\widehat{D}$ such that $W = W' \cap F$ and for each $i$, $V_i =
V'_i \cap F$. Moreover, the lemma shows the value groups of $W$ and
$W'$ are the same, as are the value groups of $V_i$ and $V'_i$, for
each $i$.  Let  $H' = V'_1 \cap \cdots \cap V'_t \cap
\widehat{D}_f$. For each $i$, let $P'_i = {\ff M}_{V'_i} \cap H'$.
   Then by the assumption that $H'$ is an essentially $n$-valuated subring
    of $\widehat{D}_f$ and  Proposition~\ref{representation2},
     $P'_1 + P'_2 + \cdots + P'_{n+1} = H'$, so there exist
     $h_1,h_2,\ldots,h_{n+1} \in \widehat{D}$ and $s\geq e$ such
      that
for each $i$,  $h_i/f^s \in P'_i$, and  $1 = \frac{h_1}{f^s} +
\frac{h_2}{f^s}+ \cdots + \frac{h_{n+1}}{f^s}$.

    Next we note that when considering the valuation rings $W',V'_1,\ldots,V'_t$, we may replace $h_1,h_2,\ldots,h_{n+1}$  with elements of $D$ that behave the same way with respect to the corresponding valuations.
This is done by applying Lemma~\ref{approximation} to the collection
of hidden prime divisors $\{w',v'_1,\ldots,v'_t\}$.  Namely, for
each
     $j =1,2,\ldots,n+1$, the lemma shows that there exists $d_j \in D$ such that $w'(d_j) = w'(h_j)$ and $v'_i(d_j) = v'_i(h_j)$ for all
      $i=1,2,\ldots,t$.

      Now for each $j =1,2,\ldots,n+1$, we have
      $0< v'_j(h_j/f^s) = v_j(d_j/f^s)$, and for each
       $i \in \{1,2,\ldots,t\} \setminus \{j\}$,
       $0 \leq v'_i(h_j/f^s) = v_i(d_j/f^s)$.
      Therefore, $d_1/f^s,d_2/f^s,\ldots,$ $d_{n+1}/f^s \in V_1 \cap
      \cdots \cap V_t \cap D_f = H$, and for each $j=1,2,\ldots,n+1$, we have
      $d_j/f^s \in {\ff M}_{V_j} \cap H = P_j$.  Now
      $P_1 + P_2 + \cdots + P_{n+1} \subseteq M \subseteq {\ff M}_{W'}$,
       so for each $j=1,2,\ldots,n+1$,
        $0 < w'(d_j/f^s) = w'(h_j/f^s)$.
         But $1 = \frac{h_1}{f^s} + \frac{h_2}{f^s} + \cdots
          + \frac{h_{n+1}}{f^s}$, so that at least one of $h_j/f^s$
          is a unit in $W$, a contradiction.  Therefore
           $P_1+P_2 + \cdots+P_{n+1} = H$, and the proof is complete.
\end{proof}

\section{Two-dimensional regular local rings}

\label{rlr section}

Until to this point, we have exhibited no examples of two-dimensional integrally closed local domains $D$ having a nonunit $f$ such that for some $n>0$, every ring between $D$ and $D_f$ is essentially $n$-valuated.  We remedy this in the present section by restricting to the case where $D$ is a regular local ring.
We apply the results of the previous sections in the proof of  Theorem~\ref{second theorem} to show that for the regular local rings we consider, we may as well assume that $D$ is a localization of $V[X]$, where $V$ is a DVR.  This case is addressed in the next lemma, but the full proof of the lemma is lengthy and requires different techniques than those employed so far.  So we postpone the crucial part of the proof of the lemma to the next section and treat it separately.  Specifically, the lemma depends on Theorem~\ref{localization}, the proof of which is the goal of  Section 7.

\begin{lem} \label{MacLane} Let $V$ be a DVR with quotient field $F$, and let $X$ be an indeterminate for
$F$. If $S$ is a multiplicatively closed subset of $V[X]$, then
every ring between $V[X]_S$ and $F[X]_S$ is
essentially one-valuated.
\end{lem}

\begin{proof} Let $H$ be an integrally closed domain with $V[X]_S
\subseteq H \subseteq F[X]_S$, and let $P$ be a prime ideal of $H$.
Define $B = H \cap F[X]$.  Then by Theorem~\ref{localization}, $B$ is an essentially one-valuated subring of $F[X]$.  Thus for $Q:=P \cap B$, there exists a
valuation overring $W$ of $B$ such that $B_Q = W \cap F[X]_Q$.
Localizing with respect to $S$, we have $(B_S)_Q = W_S \cap
(F[X]_S)_Q$.  Now $B_S = H_S \cap F[X]_S = H$, so $H_Q = W_S \cap
(F[X]_S)_Q$.  Therefore, since $H_Q \subseteq H_P$, localizing $H_Q = W_S \cap
(F[X]_S)_Q$
with respect to $P$ yields $H_P = (W_S)_P
\cap (F[X]_S)_P$, which since $(W_S)_P$ is a valuation domain proves that $H$ is an essentially one-valuated subring of
$F[X]_S$.
\end{proof}

We use the lemma to prove now  Theorem~\ref{second theorem}.
\smallskip

{\it Proof of Theorem~\ref{second theorem}.}
 Let $D$ be a regular local ring of Krull dimension $2$,
and let $f$ be a regular parameter of $D$. Then there exists $g \in D$ such that $(f,g)$ is the maximal ideal of $D$.  
Suppose  either  $D$ is equicharacteristic  or  $D$ has mixed characteristic and $f$ is a prime integer in $D$.
We prove that every ring $H$ between $D$ and $D_f$ is  an essentially one-valuated
subring of $D_f$.
First we claim that  there exists a DVR  $V \subseteq \widehat{D}$ with maximal ideal $fV$ such that $g$ is transcendental over the quotient field  of $V$ and  $V[g]_{(f,g)}$ has   completion $\widehat{D}$ in the $(f,g)$-adic topology.
If $D$ and
its residue field have the same characteristic, then since
$\widehat{D}$ is a regular local ring, the Cohen Structure Theorem
 shows that $\widehat{D} = K[[f,g]]$
for some subfield $K$ of $\widehat{D}$, and $f$ and $g$ are
analytically independent over $K$; cf.~\cite[Theorems 9 and
15]{Cohen}.  Thus with $V = K[f]_{(f)}$, then  $V[g]_{(f,g)}$ has completion $\widehat{D}$.
On the other hand, suppose $D$ has
characteristic $0$ but its residue field has characteristic $p \ne
0$ and $f=p$. In this case, applying \cite[Theorems 9 and 15]{Cohen}, we have that $\widehat{D} = V[[g]]$, where $V$ is a
complete DVR with maximal ideal $fV =pV$ and the mapping $V[[X]] \rightarrow
V[[g]]:X \mapsto g$ is a ring isomorphism.  Then $V[g]_{(f,g)}$ has completion $\widehat{D}$.

We have established that
there exists a  DVR  $V \subseteq \widehat{D}$ with maximal ideal $fV$ such that $g$ is transcendental over the quotient field  of $V$ and  $A:=V[g]_{(f,g)}$ has   completion $\widehat{D}$ in the $(f,g)$-adic topology.
  Then $A_f =
F[g]_S$, where $S = V[g] \setminus {(f,g)}$,
 and since $g$ is transcendental  over $F$, we may apply Lemma~\ref{MacLane} to obtain
that every ring between $A$ and $A_f$ is an
essentially one-valuated subring of $A_f$. By Theorem~\ref{pass to completion},
every  ring between $\widehat{D}$ and
$\widehat{D}_f$ is an essentially one-valuated subring of $\widehat{D}_f$, and another application of this same theorem shows then that every
ring  between $D$ and $D_f$ is an essentially one-valuated
subring of $D_f$. This proves Theorem~\ref{second theorem}. $\:\:\:\:\:\square$
\smallskip

If $D$ is as in Theorem~\ref{second theorem},   and $(f,g)D$ is the maximal ideal of $D$, then the theorem implies that for $m,n>0$, the rings $D_1 = \overline{D[f^m/g^n]}$ and $D_2=\overline{D[g^n/f^m]}$ are essentially one-fibered.
Using an argument due to Heinzer and Lantz, it is possible to prove something stronger.
In \cite{Go}, G\"ohner proves that if $D$ is a complete normal local Noetherian domain with torsion divisor class group, then for every hidden prime divisor $D$, there is an ${\ff m}$-primary ideal for which $V$ is the unique Rees valuation ring.  (The latter condition is known as property (N) in the literature.)  In \cite{HL}, Heinzer and Lantz give a different proof of this theorem, and in the course of the proof  they give an argument from which can be deduced the following proposition.  Since this proposition is not stated explicitly in \cite{HL}, we reproduce the relevant part of the proof.

\begin{prop} \label{Heinzer-Lantz} {\em (Heinzer--Lantz \cite[proof of Theorem 11]{HL})} Let $D$ be a complete local Noetherian domain of Krull dimension $2$ with maximal ideal ${\ff m}$, and suppose that $0 \ne f \in D$ has prime radical.  Let $H$ be the integral closure of $D[f/g]$, where $(f,g)D$ is an ${\ff m}$-primary ideal of $D$.   Then there exists a maximal ideal $M$ of $H$ such that every exceptional prime ideal of the extension $H/D$ is contained in $M$.
\end{prop}

\begin{proof}   Let $M_0$ be the maximal ideal of $D[f/g]$ generated by ${\ff m}$ and $f/g$.  We show first there is a unique maximal ideal $M$ of $H$ lying over $M_0$.  It is enough to show that the completion of $D[f/g]_{M_0}$ has local integral closure
\cite[Lemma 23.2.7.1, p.~219]{EGA}, and to prove this it suffices to show that this completion is a
domain \cite[(30.3) and (43.12)]{Nagata}.  Now $D[f/g] \cong D[X]/(gX-f)$, so since $D$ is complete, the completion of $D[f/g]_{M_0}$ is $D[[X]]/(gX-f).$ Thus we claim that $gX-f$ is a prime element of $D[[X]]$.  Since $D[[X]]$ is the intersection of the rings $D_{\ff p}[[X]]$, where ${\ff p}$ ranges over the height $1$ prime ideals of $D$, it is enough to verify that there is exactly one height $1$ prime ideal ${\ff p}$ of $D$  such that $gX-f$ is not a unit in $D_{\ff p}[[X]]$ and for this choice of ${\ff p}$, $(gX-f)D_{\ff p}[[X]]$ is a prime ideal of $D_{\ff p}[[X]]$.  Suppose that ${\ff p}$ is the unique prime ideal of $D$  that contains $f$.  Then since $(g,f)D$ is ${\ff m}$-primary, it must be that $g \not \in {\ff p}$.  Hence $g$ is a unit in $D_{\ff p}$, so that $gX-f$ is a prime element of $D_{\ff p}[[X]]$.  Otherwise, if ${\ff p}$ is a height $1$ prime ideal of $D$ that does not contain $f$, then $f$ is a unit in $D_{\ff p}$, so that  $gX-f$ is a unit in $D_{\ff p}[[X]]$.  Therefore, $gX-f$ is a prime element of $D[[X]]$, and there exists a unique maximal ideal $M$ of $H$ lying over $M_0$.  Now let $P$ be a height $1$ prime ideal of $H$ containing ${\ff m}$.  Then since ${\ff m}D[f/g]$ is a prime ideal of $D[f/g]$, and $H$ is integral over $D[f/g]$, it must be that $P \cap D[f/g] = {\ff m}D[f/g] \subseteq M_0 \subseteq M$,  By Going Up, there exists a maximal ideal  of $H$ lying over $M_0$ and containing $P$.  But $M$ is the unique maximal  ideal of $H$ lying over $M_0$, so $P \subseteq M$, and this proves the proposition.
\end{proof}





\medskip

As another corollary of the proposition, we give a class of one-fibered ideals of regular local rings. 

\begin{cor} \label{one-fibered corollary} Let $D$ and $f$ be as in Theorem~\ref{second theorem}.   If $g \in {\ff m}$ such that $g$  is prime in $\widehat{D}$ and $f \not \in gD$,  then for each $m,n>0$, $(f^m,g^n)$ is a one-fibered ideal of $D$.
\end{cor}

\begin{proof} The number of Rees valuation rings of an ${\ff m}$-primary ideal $I$ of $D$ is the same as the number of Rees valuation rings of $I\widehat{D}$
\cite[Theorem 5.3]{KV}.  Thus we may assume without loss of generality that $D$ is a complete local ring.
 Let $H$ be the integral closure of $D[g^n/f^m]$, and let $P_1,\ldots,P_t$ be the exceptional prime ideals of the extension $H/D$.
Since $f \not \in gD$ and $g$ is prime in $D$, then $(f,g)$ is an ${\ff m}$-primary ideal of $D$. Thus, as discussed in (2.4), all the Rees valuation rings occur as exceptional prime divisors of the extension $H/D$.
 By Proposition~\ref{Heinzer-Lantz},  there exists a maximal ideal $M$ of $H$ such that $P_1 + \cdots +P_t \subseteq M$.  Now since $D$ is an integrally closed ring, it follows that $D \subseteq H \subseteq D_f$.
 Thus by Theorem~\ref{second theorem}, $D[g^n/f^m]$ is an essentially one-fibered subring of $D_f$, and so the extension $H_M/D$ has only one exceptional prime ideal.  Therefore, $t = 1$ and $(f^m,g^n)$ is one-fibered.
\end{proof}

In Proposition 4.1 of \cite{Swanson}, Swanson proves that if $D$ is a local Noetherian domain with maximal ideal ${\ff m}$ of dimension $d>1$ and $(f_1,\ldots,f_d)$ is an ${\ff m}$-primary ideal of $D$, then for each $N>0$, there exists $m>N$ and a  Rees valuation ring of $(f_1^m,f_2,\ldots,f_d)$ that is not a Rees valuation ring of any ideal $(f_1^k,f_2,\ldots,f_d)$, $k < m$.   Thus in the context of Corollary~\ref{one-fibered corollary}, the ideals $(f^m,g^n)$ are one-fibered, but there are infinitely many DVRs that arise as the Rees valuation rings of these ideals.

\section{Proof of Lemma 6.1}
\label{MacLane section}

Our arguments in Section 6 reduce the proof of our main result,
Theorem~\ref{second theorem}, to the special case where $H$ is an integrally
closed ring between localizations of the two-dimensional regular
local ring $V[X]$ and the PID $F[X]$, where $V$ is a DVR with
quotient field $F$, so that what remains to be shown is that every
ring between $V[X]$ and $F[X]$  is essentially one-valuated.  We
complete this last  step of the proof  with a  technical analysis of
the valuation theory for this setting.  Using
MacLane's notion of key polynomials as developed in \cite{lane}, our
arguments follow somewhat closely those in \cite{lt}, where similar
results were proved for the case  $V = {\mathbb{Z}}_{(p)}$, where
$p$ is a prime integer. However, we include additional details and
clarify a few points from \cite{lt}.

Throughout this section, $V$ denotes a DVR with maximal ideal $M = \pi V$, quotient field $F$
and residue field $E_0$ (we do not assume any restriction on $F$ and
$E_0$).   We denote by $v$ the corresponding valuation on $F$, and we assume that $v$ has value group ${\mathbb{Z}}$, and hence that $v(\pi) = 1$.
Following the terminology and notation of \cite{lt}, we
 refer to the valuation overrings of $V[X]$ in which $\pi$
is a nonunit as \textit{$\pi$-unitary}, or simply \textit{unitary}.

The basic tool used to study $\pi$-unitary valuation overrings of
$V[X]$ is a technique for extending valuations commuted from
MacLane's paper \cite{lane}. We will describe this procedure quite
in detail, though not proving all the results, and adapt it to our
situation.

\begin{definition} \label{first stage} Given a  positive real number $\mu_1$, a \textit{first-stage
extension} $V_1$ of $V$ to $F(X)$  is the valuation ring associated to the
valuation $v_1$  defined by:
 \begin{itemize}

 \item $v_1(X):= \mu_1$;

\item for each $f(X)= \sum_{i=0}^na_iX^i \in F[X]$, $v_1(f(X)) :=
 \min\{v(a_i)+i\mu_1\}$.
 \end{itemize}
 \end{definition}

It is easy to check that $V_1$ is a $\pi$-unitary, rank-one
valuation domain containing $V[X]$ (this last property is  due to the
fact that $\mu_1 > 0$) and that $v_1(a) = v(a)$ for each $a \in F$.
Obviously, $V_1$ is a DVR if and only if $\mu_1$ is rational;
otherwise, $V_1$ is a rank-one nondiscrete valuation domain.

Now, it is possible to extend $v_1$ in order to obtain another
$\pi$-unitary valuation $v_2$ with associated valuation ring $V_2$,
such that $V_1 \cap F[X] \subseteq V_2 \cap F[X]$ and also $M_1 \cap
F[X] \subseteq M_2 \cap F[X]$, where $M_i$ is the maximal ideal of
$V_i$, for $i=1,2$. To do this we need to introduce MacLane's
concept of  \textit{key polynomial} (cfr. \cite[\S 3, 4 and
9]{lane}).   To formulate the definition, we require a
  divisibility relation for elements of $F(X)$:

\begin{definition} Let $w$ be a valuation on $F(X)$ with associated valuation ring $W$, and let  $a, b \in F(X)$.  Then
  $b$ is \textit{equivalence-divisible} by
 $a$ in $W$ if and only if there exists   $c \in F(X)$ such that
 $w(b)=w(ac)$.
 \end{definition}

 \begin{definition} \label{key poly}
For $w$ a valuation on $F(X)$ with valuation ring $W$,
 a polynomial $\phi(X) \in F[X]$ is  a
\textit{key polynomial} over $W$ if all of the following conditions hold.

\begin{enumerate}
\item The leading coefficient of $\phi$ is $1$.

\item If a product of polynomials   $f_1 \cdots f_s \in F[X]$ is
equivalence-divisible  by $\phi$ in $W$, then one of the $f_i$'s is
equivalence-divisible  by $\phi$ in $W$.

\item For any nonzero polynomial $f(X) \in F[X]$ that is equivalence-divisible  by
$\phi$ in $W$ it is the case that $\deg(f) \geq \deg(\phi)$.
\end{enumerate}
\end{definition}

Point (a) has the important consequence that it guarantees that each nonzero polynomial $f(X) \in F[X]$
can be written as:
$$f(X) = a_n(X)(\phi(X))^n +
a_{n-1}(X)(\phi(X))^{n-1} + \cdots + a_1(X)(\phi(X)) + a_0(X), \quad
\textrm{(I)}
$$

{\noindent}where $a_i(X) \in F[X]$ and $\deg(a_i(X)) < \deg(\phi(X))$.

\bigskip

Returning to the first-stage inductive valuation ring $V_1$ from
Definition~\ref{first stage},  MacLane describes a process by which
a choice of key polynomial $\phi \in F[X]$ and real number $\mu$ is
used to extend the valuation $v_1$ to a valuation $v_2$ as follows
\cite[Theorem 4.2]{lane}.
 \begin{enumerate}

 \item Set $v_2\mathrm{}(\phi):= \mu$.

\item  For each nonzero polynomial $f(X) \in F[X]$, put
$$v_2(f(X)) := \min_{i = 1, \ldots, n}\{v_1(a_i(X)) + i\mu \},$$
{\noindent}where the $a_i(X) \in F[X]$ are the coefficients of the
expansion of $f(X)$ in $\phi$ of (I).

\end{enumerate}

Points (b) and (c) of Definition~\ref{key poly} are used to prove that
$v_2(fg)=v_2(f) + v_2(g)$, for each $f,g \in F[X]$, and to show the
 \textit{monotonicity property} (\cite[Lemma 4.3 and
Theorem 5.1]{lane}):
$$v_2(f) \geq v_1(f) , \quad \textrm{for each}  \quad f \in F[X]\backslash \{0\},$$

\noindent (in particular, $v_2(f) = v_1(f)$ if $\deg(f) \leq
\deg(\phi)$).

Monotonicity   implies that $V_1 \cap F[X] \subseteq V_2 \cap F[X]$
and that $M_1 \cap F[X] \subseteq M_2 \cap F[X]$, where $M_1$ and
$M_2$ are respectively the maximal ideals of $V_1$ and $V_2$.

Inductively, iterating the procedure described above, it is possible
to construct sequences of valuation domains $V_1, V_2, \ldots, V_k$,
where each $V_i$ is obtained by extending $V_{i-1}$,  fixing a key
polynomial $\phi_i$ over $V_{i-1}$ and a real positive value
$\mu_i$. In this case we  write $V_i = (V_{i-1}, \phi_i, \mu_i)$.
 MacLane calls $V_i$ an \textit{augmented value} of $V_{i-1}$.


\begin{definition}\label{k-stage}
{\em With the notation and hypotheses given above, the
valuation overring $W$ of $V[X]$ is   \textit{$k^{th}$-stage
inductive} or, more generally, an \textit{inductive valuation ring},  if
there is a sequence of valuation domains $\{V_1, \ldots,
V_{k-1},V_k=W\}$, where $V_1$ is a first-stage extension of $V$,
each  $V_{i}$ is
 an augmented value of $V_{i-1}$,
for $i=2, \ldots, k$ such that $\mu_1,\ldots,\mu_{k-1}$ are positive
rational (so $V[X] \subset V_i$) and  the key polynomials $\phi_i$,
for $i = 1,\ldots, k$, satisfy:
\begin{enumerate}
\item $\deg(\phi_{i+1}) \geq \deg(\phi_{i})$;

\item $\phi_{i+1}$ is not equivalence divisible in $V_i$ by
$\phi_i$.
\end{enumerate}
The sequence $\{V_1, \ldots,
V_{k-1},V_k\}$ is an \textit{inductive sequence}.
If $\mu_k$ is rational, then $V_k$ is {\it inductive commensurable};  otherwise $V_k$ \textit{inductive incommensurable}.}
\end{definition}

It directly follows from the definition that first-stage or, more
generally, the inductive valuations are $\pi$-unitary  since their
values on elements of $F$ are the exact values of $v$, whence $\pi$
is a nonunit. Moreover, it is easy to
check that the inductive commensurable valuations are DVRs, while   the incommensurable are one-dimensional  but   not discrete.

 With
the above notation,  if two consecutive key polynomials $\phi_i$ and
$\phi_{i+1}$ have the same degree, then $V_{i+1} = (V_{i-1},
\phi_{i+1}, \mu_{i+1})$ \cite[Lemma 15.1]{lane}. So one may always restrict consideration to  finite inductive sequences in which the degrees of the
key polynomials are strictly increasing \cite[Theorem 15.2]{lane}.

\bigskip
Each valuation ring in an inductive sequence is conditioned on a key
polynomial, and hence a method for exhibiting key polynomials is
needed in order to form inductive sequences. MacLane gives a way to
find key polynomials over a $k^{th}$-stage inductive commensurable
valuation domain. We  sketch  his method, at each step referring  to
\cite{lane} for the proofs.  The method requires the following facts
about the residue fields of valuation rings in an inductive
sequence.

\begin{lem} \label{resfieldincommensurable}\label{resfieldcommensurable} \textrm{\cite[Theorems 10.2, 12.1 and 14.2]{lane}}
Let  $\{V_1, \ldots, V_{k-1},V_k\}$ be an inductive sequence of valuation
overrings of $V[X]$, and let $M_i$ and $E_i$ denote the maximal ideal and residue field of $V_i$, respectively.    Then:
\begin{itemize}

\item[(1)]  $E_k$  is algebraic over $E_0$.

\item[(2)] If $V_1$ is commensurable (which is the case if $k>1$), then  $$(V_1 \cap F[X])/(M_1 \cap F[X])
\cong E_0[Y],$$ where $Y$ is an indeterminate for $E_0$.

\item[(3)]  If
  $V_k$ is commensurable then there exists a
sequence of algebraic field extensions $E_0 \subseteq E_1 \subseteq
E_2 \subseteq  \cdots \subseteq E_k$ such that for each $i=1,\ldots,k$, $$(V_i \cap F[X])/(M_i
\cap F[X]) \cong E_i[Y],$$  where $Y$ is an indeterminate for $E_i$. \end{itemize}
\end{lem}

We also need the following lemma.

\begin{lem}\label{radical}
Let $W$ be a $\pi$-unitary valuation overring of $V[X]$ and $M_w$ be
the maximal ideal of $W$.  Let $D := W \cap F[X]$   and $P := M_w
\cap D$ (this ideal is called the valuation prime of $D$). Then the
radical of the ideal $\pi D$  is a prime ideal of $D$. If  the value
group of $W$ is contained in the additive group of rational numbers,
then $D_P = V$.
\end{lem}

\begin{proof}
 Note that  $\pi$ is a unit in $F[X]$, so the  prime ideals of
 $D$ containing $\pi$ do not lift to $F[X]$. Hence, since    $M_w$ is the
radical in $W$ of the principal ideal $(\pi)$, it follows that $M_w
\cap D$ is  is the radical of $\pi D$ in $D$ and it is a prime
ideal.  If also the value group of $W$ is contained inthe field of
rational numbers, then since $F[X] \not \subseteq W$, it follows
that $D_P = V$ \cite[Lemma 1.3]{HO}.%
\end{proof}

 Using the lemmas, we show now how to construct a key polynomial in order to extend an inductive sequence.   Let $W$ be a $k^{th}$-stage inductive commensurable domain, and let $M_w$ be the maximal ideal of $W$.
  \begin{enumerate}

 \item By Lemma~\ref{resfieldcommensurable},  there exists a field $E$ such that $$(W \cap F[X])/(M_w \cap F[X]) \cong E[Y].$$

\item  For $\m$  a height-two prime ideal  of $W \cap F[X]$, we have
  $\m \supset M_w \cap F[X]$. In fact, $W \cap F[X]  \supset V$ and $\m \cap V \neq (0)$, since $\m$ is not upper
  to zero (otherwise it would be height one). Thus $\m \cap V = M =\pi
  V$. By Lemma \ref{radical}, $M_w \cap F[X]$ is
  the radical of  the ideal   $\pi(W \cap F[X])$ and so $\m \supset M_w \cap
  F[X]$.

  \item  Since  $M_w \cap F[X]$ is
  the radical of  the ideal   $\pi(W \cap F[X])$, then $\m/(M_w \cap F[X])$ corresponds to a height one
  prime ideal of $E[Y]$ and so to an irreducible polynomial $\psi(Y) \in
  E[Y]$.

  \item In \cite[Theorem 13.1]{lane} it is shown that, for a fixed irreducible polynomial $\psi \in E[Y]$, there exists a unique
   (modulo equivalence-divisibility in $W$) key polynomial $\phi$ over $W$  such that $\varphi(f \cdot \phi) = \psi$,
   where  $f$ is a suitable polynomial (see \cite[Lemma 11.1]{lane}) and $\varphi$ is the canonical projection:
   $$\varphi: W \cap F[X] \twoheadrightarrow W \cap F[X]/M_w \cap F[X] \cong E[Y].$$

   \item
 By extending $W$ using $\phi$
  and an assigned value $\mu$, we get an augmented valuation domain $W'$
  such that $M_{w'} \cap W \cap F[X] = \m$.
\end{enumerate}

Another  concept introduced in MacLane's paper is that of
\textit{limit valuation}.

\begin{definition} {\em Let be given an infinite sequence of  inductive commensurable
valuation overrings of $V[X]$, $\{V_k\}_{k \geq 0}$ (\cite[page
372]{lane}). We define $v_{\infty}(f) := \lim_{k \rightarrow
\infty}v_k(f)$, for each $f \in F[X]$. By the usual extension to
$F(X)$ (i.e., $v_{\infty}(f/g)= v_{\infty}(f)-v_{\infty}(g)$, $f,g
\in F[X], \, g \neq 0$), $v_\infty$ is a $\pi$-unitary valuation on
$F(X)$ \cite[Theorem 6.2]{lane}.  The valuation $v_\infty$  is
called {\it limit valuation}.  It is a {\it finite limit valuation}
if the only element of $F[X]$ having value $\infty$ is $0$;
otherwise, $v$ is an {\it infinite limit valuation}.}
\end{definition}

 As we will see, the finite limit valuation domains are DVRs and the infinite limit valuation domains are two-dimensional valuation domains with $\textrm{height}(\pi) = 2$.

Note that from the monotonic property of
inductive values, $v_{k+1}(f) \geq v_k(f)$ for each nonzero
polynomial $f \in F[X]$, so either  $\lim_{k \rightarrow \infty}v_k(f)$ exists or is $\infty$.

By \cite[Theorem 15.3]{lane} a commensurable inductive valuation
domain cannot be realized as a limit of an infinite sequence
$\{V_k\}_{k \geq 0}$. Thus the classes of limit valuation domains
and inductive valuation domains are disjoint.



\begin{prop}\label{finitelimit is DVR}
If $V_\infty$ is a finite limit valuation domain, then $V_\infty$
has rational value group.
\end{prop}

\begin{proof}
Let be given a finite limit valuation domain $V_\infty = \lim_k V_k$, where
$\{V_k\}_{k \geq 0}$ is an infinite sequence of inductive
commensurable valuation domains.  We show that the value group $\Gamma_\infty$ of $V_\infty$ is
rational.

By hypothesis the value groups $\Gamma_k$ of the $V_k$'s are
rational and, in particular, they are cyclic.
As regards the key polynomials $\phi_k$, there are two possibilities:

\begin{enumerate}
\item the sequence $\{\deg(\phi_k)\}_{k \geq 0}$ doesn't stabilize;

\item the sequence $\{\deg(\phi_k)\}_{k \geq 0}$  stabilizes.
\end{enumerate}

In the first case, by \cite[Theorem 5.1]{lane}, for any $f(X) \in
F[X]$ the sequence ${v_k(f)}_{k>0}$ stabilizes (it is enough that
the degree of $\phi_k$ exceeds the degree of $f(X)$). Thus,
$v_\infty(f) \in \cup_k \Gamma_k$, and so it is rational.

In the second case we may assume that $\{\deg(\phi_k)\}_{k \geq 0}$
stabilizes at some point $k=t$.
Then, for every $j>t$, by \cite[Lemma 6.3(i)]{lane} we have that
$\mu_j = v_t( \phi_{j+1} - \phi_j )$. This means in particular that
$\mu_j \in \Gamma_t$.
Again by \cite[Theorem 6.6]{lane}, $\Gamma_j$ is generated as a free
group by $1,\mu_1,...,\mu_j$, thus $\Gamma_j \leq \Gamma_t$.
Now, $\Gamma_t$ is cyclic, whence a sequence of elements  in
$\Gamma_t$ (for instance $\{v_k(f)\}_{k \geq 0}$) either stabilizes
or has limit equal to infinity. Since $V_\infty$ is a finite limit
valuation domain, this implies that $\{v_k(f)\}_{k>0}$ has limit in
$\Gamma_t$. Therefore, $\Gamma_\infty$ is cyclic and $V_\infty$ is a
DVR.
\end{proof}

The infinite limit valuations, as the name suggests,  assume also
infinite values, in which case the corresponding valuation rings
have Krull dimension $2$.   In this case, since $v_{\infty}(\pi)=
v(\pi) < \infty$,   we have that $\pi$ is not contained in the
height-one prime ideal of $V_{\infty}$.

We also observe that, for an infinite limit valuation domain
$V_\infty$, the degrees of the key polynomials $\phi_k$ cannot
increase indefinitely. In fact, as pointed out in the proof of
Proposition \ref{finitelimit is DVR}, if the sequence
$\{\deg(\phi_k)\}_{k \geq 0}$ doesn't stabilize, then the value
group of $V_\infty$ is rational, hence contradicting  the fact that
$V_\infty$ is two-dimensional.

\begin{lem}\label{height-one}
Let $V_\infty$ be an infinite limit valuation domain   and let $P$
be its height-one prime ideal. Then $(V_\infty)_P =F[X]_{(f)}$, for
some irreducible polynomial $f \in F[X]$, and the center $P_\infty$
of $V_\infty$ in $D = V_\infty \cap F[X]$ has  height 2 and it is
the only prime ideal of $D$ containing $\pi$.
\end{lem}

\begin{proof}
By construction we have that $v_\infty(\pi)=v(\pi) < \infty$. So
$\pi \in M_\infty \backslash P$, where $M_\infty$ is the maximal
ideal of $V_\infty$. Then $(V_\infty)_P \supseteq F[X]$ and so
$(V_\infty)_P =F[X]_{(f)}$ for some irreducible polynomial $f \in
F[X]$. Also, $P \cap F[X] \subseteq P_\infty$ and $\pi \notin P \cap
F[X]$, whence $P_\infty$ is height-two. Since, by Lemma
\ref{radical} $P_\infty$ is the radical of $\pi$ in $D$, it follows
that $\pi$ is the only prime ideal of $D$ containing $\pi$.
\end{proof}

%

The next result, whose proof is commuted from
\cite[Theorem 8.1]{lane}, gives a description of valuation overrings
of $V[X]$ for which $\pi$ has finite value.

\begin{prop}\label{charMaclanetype}
Each $\pi$-unitary valuation overring $W$ of $V[X]$ such that
$w(\pi) < \infty$ is a $k^{th}$-stage inductive valuation domain,
for some $k < \infty$, or it is a limit valuation
domain.
\end{prop}

\begin{proof}
Notice that $W$ can always be normalized and so we can suppose that
$w(\pi) = v(\pi)$ (whence $w(a) = v(a)$, for all $a \in F$).

Using an inductive process, we will construct a sequence of
inductive   valuation domains $\{V_1, V_2, \ldots, V_k\}$, $V_i =
(V_{i-1},\phi_i, \mu_i)$ with $V_1, V_2, \ldots, V_{k-1}$
commensurable and such that for all nonzero $f(X) \in F[X]$ the
following conditions hold:

\begin{enumerate}
\item $w(f(X)) \geq v_k(f(X))$;

\item $\deg(f(X)) < \deg(\phi_k) \Rightarrow w(f(X)) = v_k(f(X))$;

\item $w(\phi_i(X)) \geq v_k(\phi_i(X)) = \mu_i$, for each $i = 1,
\ldots,k$.
\end{enumerate}

We start by defining $V_1$ as a first-stage inductive valuation
domain with $\phi_1 := X$ and $\mu_1 := w(X)$. If $w(f(X)) =
v_1(f(X))$, for each $f(X) \in F[X]$, then $W = V_1$ and we have
done.  If not,  we construct $V_1$ by taking a rational value $\mu_1
\leq w(X)$ (actually, we put $\mu_1 = w(X)$ if $w(X)$ is rational),
so that $V_1$ is inductive commensurable. It is easy to check that
$V_1$ satisfies conditions (1)-(3).

Suppose we have constructed a commensurable inductive value $V_i$
satisfying the above properties (1)-(3) and such that $W \neq V_i$.
Then there exists a polynomial $f(X) \in F[X]$ such that $w(f(X)) >
v_i(f(X))$. Choose $\psi(X) \in F[X]$ of minimal degree among the
polynomials with $w(\psi(X))
> v_i(\psi(X))$. It is always possible to choose $\psi(X)$ having leading
coefficient $1$.

Following exactly the same argument used in \cite[Theorem
8.1]{lane}, we have that the following two conditions are equivalent
for each $f(X) \in F[X]$:

\begin{enumerate}
\item $w(f(X)) > v_i(f(X))$

\item $f(X)$ is equivalence-divisible by $\psi(X)$ in $V_i$.

\end{enumerate}

The equivalence of (a) and (b) gives directly that $\psi(X)$ is a
potential key polynomial over $V_i$ and so it can be used to
construct an augmented value $V_{i+1}$ of $V_i$. If putting
$\mu_{i+1} = v_{i+1}(\psi(X)) = w(\psi(X))$ we get that $W =
V_{i+1}$, then we are done. If not,   we put $\mu_{i+1} :=
w(\psi(X))$ if the latter is a rational value,
otherwise (if
$w(\psi(X))$ is irrational) we put $\mu_{i+1}$ to be a fixed
rational value such that $v_i(\psi(X)) < \mu_{i+1} < w(\psi(X))$.
The proof that $V_{i+1}$ satisfies conditions (1)-(3) is exactly as
in \cite[Theorem 8.1]{lane}.

So the process will finish in the case   there is an integer $k$
such that $W = V_k$, or it will give an infinite sequence of
commensurable inductive valuation domains $\{V_i\}_{i \geq 0}$ such
that
$$w(f(X)) \geq v_\infty(f(X)) = \lim_{i}v_i(f(X)).$$

We show that $w(f(X)) = v_\infty(f(X))$  for all $f(X) \in F[X]$ and
so $W=V_\infty$.

Suppose, by contradiction, that there exists a nonzero polynomial
$f(X) \in F[X]$ such that $w(f(X)) > v_\infty(f(X))$. Consider the
sequence $\{v_i(f(X))\}_{i \geq 0}$, that is monotone
non-decreasing. Then $w(f(X)) > v_i(f(X))$, for all $i \geq 0$. This
implies that $f(X)$ is equivalence-divisible by $\phi_{i+1}$ in
$V_i$ (from the above equivalence of points (a)-(b), where $\psi =
\phi_{i+1}$). Then, by \cite[Theorem 5.1]{lane}, $v_{i+1}(f(X)) >
v_i(f(X))$ and this implies  (by (2)) that the degrees of the key
polynomials stabilize. So there exists an integer $t$ such that
$\deg(\phi_i) = M$, for each $i > t$. This implies that $\Gamma_t =
\Gamma_i$ for each $i > t$ (\cite[page 376]{lane}). So
$\{v_i(f(X))\}_{i \geq 0} \subset \Gamma_t$, that is a cyclic group,
whence $w(f(X)) \geq \lim_{i}v_i(f(X)) = \infty$. This implies that
$f(X)=0$, against the assumption on $f(X)$. It follows that
$W=V_\infty$.
\end{proof}

%
%


 \begin{definition}\label{order}
 {\em  Given two $\pi$-unitary valuation overrings of $V[X]$, $V_1,V_2$  with maximal ideals
  $M_1$ and $M_2$ respectively, we say that $V_1
  \preceq V_2$ if:
  $$V_1 \cap F[X] \subseteq V_2 \cap F[X] \quad \textrm{and} \quad M_1 \cap F[X]
  \subseteq M_2 \cap F[X].$$}
  \end{definition}


We  will show in Corollary~\ref{maximal} that the limit valuation
domains are maximal elements in the set of the $\pi$-unitary
overrings of $V[X]$, with respect to the relation $\preceq$ defined
above. More precisely, if $W$ is a limit valuation domain and $W'$
is another $\pi$-unitary valuation domain (both with quotient field
$F(X)$) such that $W \preceq W'$, then $W=W'$.

\begin{lem}\label{max finite_val_prime}
Let $V_\infty$ be a finite limit valuation domain.  Then the residue
field of $V_\infty$ is algebraic over the residue field $E_0$ of
$V$. Let $P_\infty$ be the center of $V_\infty$ in $D = V_\infty
\cap F[X]$. Then $P_\infty$ is a maximal ideal of $D$ and it is
height-one (whence it is the only prime ideal of $D$ containing
$\pi$).
\end{lem}

\begin{proof}
The first part of the statement is completely proved in
\cite[Theorem 14.1]{lane}. For the second part, it is sufficient to
observe that $E_0 \subset D/P_\infty \subset V_\infty/P_\infty$, so
$D/P_\infty$ is a field, since the extension $E_0 \subset
V_\infty/P_\infty$ is algebraic. By Proposition \ref{finitelimit is
DVR} the value group of $V_\infty$ is rational and so $D_{P_\infty}
= V_\infty$, whence $P_\infty$ is height-one and the thesis follows.
\end{proof}


\begin{prop}\label{finite limit prufer} \label{infinite limit prufer}
If $V_\infty$ is a limit valuation domain, then $D :=
V_\infty \cap F[X]$ is a Pr\"ufer domain.
\end{prop}

\begin{proof}
Suppose first that $V_\infty$ is a finite limit valuation domain.
If $P \in \Spec(D)$ is such that $P \cap V = (0)$, then $D_P =
F[X]_{D \backslash P}$ which is a DVR.
Otherwise, if $P \cap V = M = (\pi)$, then $P = P_\infty$, which by Proposition~\ref{max finite_val_prime} is the only
prime ideal of $D$ containing $\pi$, and $D_{P_\infty} = V_\infty$,
by Lemma \ref{radical}.

Next, suppose that $V_\infty$ is an infinite limit valuation domain.
If $P \in \Spec(D)$ is such that $P \cap V = (0)$, then as in the
finite limit case, $D_P$ is a DVR. So suppose that  $P \cap V = M =
(\pi)$, so that by Lemma~\ref{height-one}, $P = P_\infty$. Let $f,g
\in D$. If $v_\infty(f) < \infty$ and/or $v_\infty(g) < \infty$,
following the arguments used to prove the second part of \cite[Lemma
1.3]{HO}, we have that $f/g$ or $g/f$ is in $D_{P_\infty}$.

Consider now the case in which $v_\infty(f) = v_\infty(g) = \infty$.
Then, following the same argument of \cite[Lemma 1.23]{lt}, we have
that $f$ and $g$ have a common factor. So, as the problem being to
see whether $f/g$ or $g/f$ are in $D_{P_\infty}$, we can easily
reduce, without loss of generality,  to the case in which
$v_\infty(f) = \infty$ and $v_\infty(g) < \infty$.

Since $v_\infty(g) > 0$, there exists positive integers $n,r$ such
that $v_\infty(g^n/ \pi^r) = 0$. This implies that $g^n/ \pi^r \in D
\backslash P_\infty$. Moreover, $v_\infty(f^n/ \pi^r) = \infty$, so
$f^n/ \pi^r \in D$. Thus $(f^n/ \pi^r)/(g^n/ \pi^r) \in
D_{P_\infty}$. It follows that $(f/g)^n \in D_{P_\infty}$ and, since
$D_{P_\infty}$ is integrally closed, $f/g \in D_{P_\infty}$.
Thus $D_{P_\infty}$ is a valuation domain and so $D$ is
Pr\"ufer.
\end{proof}

\begin{cor}\label{maximal}
Let $V_\infty$ be a limit valuation domain and $W$ be   a
$\pi$-unitary valuation domain  such that $V_\infty \preceq W$. Then
$V_\infty=W$.
\end{cor}

\begin{proof}
Since $V_\infty \preceq W$, then $V_\infty \cap F[X] \subseteq W
\cap F[X] \subseteq W$. We have shown that $V_\infty \cap F[X]$ is a
Pr\"ufer domain (Proposition \ref{finite limit prufer}). Hence,  $W$ is a localization of
$V_\infty \cap F[X]$. But, since $W$ is $\pi$-unitary it can only be
equal to $V_\infty$.
\end{proof}

%
%

  \begin{lem}\label{monotonicity} Let $W$ be  a $\pi$-unitary valuation overring of
  $V[X]$, and $\{V_1, \ldots, V_h\}$ be an inductive sequence.
   Suppose that $v_i(\phi_i) \leq w(\phi_i)$, for each $i=1, \ldots,h$,
   where $\phi_i$ is the key polynomial used to construct $V_i$ as an augmented value of $V_{i-1}$.
   Then $V_h \preceq W $.

  \end{lem}

  \begin{proof}
  We prove the statement by induction over $i \geq 1$.

  Suppose that $v_1(X) = \mu_1 \leq w(X)$. By Proposition \ref{charMaclanetype}, $W$
  is a $k^{th}$-stage inductive valuation domain,
for some $k < \infty$, or it is a limit valuation. Thus, $w(a) =
v_1(a) =v(a)$, for each $a \in F$.

  Then, for each $f(X)=\sum_{j=0}^s a_jX^j
  \in F[X]$, we have that:

   $$v_1(f(X))= \min_{j = 1, \ldots, s}\{v(a_j) + j\mu_1 \} \leq
     \min_{j = 1, \ldots, s}\{w(a_j) + w(X^j)\}  \leq  w(f(X)).$$

Suppose by induction that $v_i(g(X)) \leq w(g(X))$, for each $g(X)
\in F[X]$ and $i < h$. Then, given $f(X)=\sum_{j=0}^r
a_j(X)\phi_{h}^j
  \in F[X]$, we have that:

$$v_{h}(f(X))= \min_{j = 1, \ldots, r}\{v_{h-1}(a_j(X)) + j\mu_{h} \} \leq
     \min_{j = 1, \ldots, r}\{w(a_j(X)) + w(\phi_{h}^j)\}  \leq  w(f(X)).$$
\end{proof}

\begin{rem}
{\em We observe that the previous result may not hold if $W$ is not
$\pi$-unitary. In fact, in this case,  $w(a) = 0$, for each $a \in
F$, and so the basis of the inductive process ($v_1(f) \leq w(f)$,
for all $f \in F[X]$) cannot be proven.}
\end{rem}

\begin{thm}\label{localization}  Every ring between $V[X]$ and $F[X]$ is an essentially one-valuated subring of $F[X]$.
\end{thm}

\begin{proof} By Proposition~\ref{n case}, it is enough to show that every finitely generated $V[X]$-subalgebra
$H$ of $F[X]$ is essentially one-valuated.  Let $R$ be the integral
closure of such a subalgebra.  Then $R  = W_1 \cap \cdots \cap W_n
\cap F[X]$, with each $W_i$ a $\pi$-unitary DVR overring of $V[X]$.
In fact, the $W_i$'s can be chosen to be the localizations of $R$ at the height one
prime ideals of $R$ lying over maximal ideals of $V[X]$, and   we may assume  that   $W_i \cap F[X] \nsubseteq W_j
\cap F[X]$, for  $i \neq j$.

Let $P \in \spec(R)$. If $P \cap V = (0)$, then $R_P$ is a
localization of $F[X]$, and hence $R_P$ trivially is an intersection
of a valuation overring and a localization of $F[X]$.  So suppose
that $P \cap V = M = \pi V$.  Then we argue as follows. Let $P_i$ be
the center of $W_i$ in $R$. Then $P$ contains at least one of the
centers $P_i$, for some $i=1,\ldots,n$. We observe that $P_i$ is the
contraction in $R$ of the valuation prime of $W_i \cap F[X]$ and,
  by Lemma~\ref{radical}, this is the radical of $(\pi)$ in $W_i \cap F[X]$. Assume that $P$
does not contain any $P_i$, for  $i=1,\ldots,n$. Thus, setting $S :=
R \backslash P $, we have that $F[X] \subseteq S^{-1}(W_i \cap
F[X])$, for each $i=1,\ldots,n$. Thus $F[X] \subseteq R_P$, against
the assumption that $P$ contains $\pi$. So $P$ contains at least one
$P_i$, for some $i=1,\ldots,n$.

Now, we claim that $P$ contains exactly one $P_i$.
Suppose by contradiction that $P$ contains  exactly $P_1,P_2,
\ldots, P_k$, for some $2 \leq k \leq n$. Then $R_{P} = (W_1 \cap
\cdots \cap W_k \cap F[X])_{R \backslash P}$ and $P_1, \ldots, P_k
\subseteq P$.
 We will show that there
exists an inductive commensurable domain   $\overline{W}$ such that
$\overline{W} \cap F[X] \subseteq R_{P}$ and at least two
different $W_i$ and $W_j$, with $i,j=1, \ldots,k$ have different
centers in $\overline{W} \cap F[X]$.

The centers $\m_1, \ldots, \m_k$ of $W_1, \ldots, W_k$ in $V[X]$ are
height-two. In fact, the only height-one prime of $V[X]$ containing
$\pi$ is $M[X]$, and if $\m_i = M[X]$ then $W_i = V[X]_{M[X]}=V(X)$.
This is not possible since we are assuming that $W_i \cap F[X]
\nsubseteq W_j \cap F[X]$, for $i \neq j$ (observe that $V(X) \cap
F[X] = V[X]$ and $V[X] \subset W_i, W_j$).

If, at least, two of the $\m_i$'s   are distinct, then they are
comaximal in $V[X]$ and so $P = R$ against the assumption.

If $\m_1 = \cdots = \m_k = \m$,  we extend $V(X)$ using a key
polynomial $\phi_1$ associated to $\m$ (in fact, for each nonzero
prime ideal of $V[X]/M[X]$ it is possible to define a key
polynomial, \cite[\S 9]{lane}) and with assigned value $\mu_1 :=
\min_{i=1,\ldots,k}\{w_i(\phi_1)\}$. We get an inductive
commensurable domain $B_1$ such that $B_1 \preceq W_i$, for
$i=1,\ldots,k$ (Lemma \ref{monotonicity}). Thus, $B_1 \cap F[X]
\subseteq W_1 \cap \cdots \cap W_k \cap F[X] \subseteq R_P$. Now, we
consider the centers of $W_1, \ldots, W_k$ in $B_1 \cap F[X]$. If,
at least, two of them are distinct, then we set $\overline{W} = B_1$
and we are done. If not, following the same procedure used to
construct $B_1$, we extend $B_1$ to an inductive commensurable
domain $B_2$ such that $B_2 \preceq W_i$, for $i=1,\ldots,k$ (also
in this case, the (unique) center of the $W_i$'s in $B_1 \cap F[X]$
is a height-two prime containing $\pi$ and we can associate to it a
key polynomial). Then, we apply to $B_2$ the same argument used for
$B_1$ and either put $\overline{W} = B_2$ or find an inductive
commensurable domain $B_3$ which is a MacLane extension of $B_2$ and
such that $B_3 \preceq W_i$, for $i=1,\ldots,k$. After finitely many
steps we will find the desired inductive commensurable domain
$\overline{W}$. For otherwise, we would have an infinite inductive
sequence $\{B_j\}_{j \geq 0}$, such that $B_j \preceq W_i$, for each
$j \geq 0$ and $i = 1,\ldots,k$. Thus $W_1 = \cdots = W_k = \lim_{j
\rightarrow \infty}B_j$, since the limit valuations are maximal
points with respect to $\preceq$ (Corollary \ref{maximal}).

Now, consider the valuation domain $\overline{W}$. By construction
$\overline{W} \preceq W_i$, for $i=1,\ldots,k$ and thus
$\overline{W} \cap F[X] \subseteq R_{P}$. Without loss of
generality, suppose that $W_1$ and $W_2$ have different centers
$\m_1$ and $\m_2$ in $\overline{W} \cap F[X]$. Then both $\m_1$ and
$\m_2$ are contained in $PR_P$. But, by construction, $\m_1$ and
$\m_2$ are distinct height-two primes in
 $\overline{W} \cap F[X]$, whence they are comaximal. So $PR_P = R_{P}$
 and this is a contradiction.

 The thesis follows.
\end{proof}

From the theorem we deduce a generalization of \cite[Corollary 3.3]{lt}, in which the case $D = {\mathbb{Z}}$ was proved.

\begin{cor}
If $A$ is a Dedekind domain with quotient field $F$, then every ring between $A[X]$ and $F[X]$ is essentially one-valuated.
\end{cor}

\begin{proof}
Let $H$ be an integrally closed ring between $A[X]$ and $F[X]$, and
let $M$ be a maximal ideal of $H$. If $M \cap A = (0)$, then $F[X]_M
= H_M$ and this is a DVR. If $M \cap A = m \neq (0)$, let $V :=
A_{(A \backslash m)}$. Then $V[X] \subseteq H_M \subseteq F_{(H
\backslash M)}[X]$ and $V[X] \subseteq H_M \cap F[X] \subseteq F[X]$. By
Theorem \ref{localization}, $H_M \cap F[X]$ is an essentially
one-valuated subring of $F[X]$. Now, $(H_M \cap F[X])_{MH_M \cap F[X]} = H_M$, and
this proves that $H$ is an essentially one-valuated subring of $F[X]$.
\end{proof}

{\it Acknowledgment}. We thank the referee for a careful reading and comments that improved the presentation of the paper.


\end{document}